\documentclass[letterpaper,12pt,oneside,final]{article}

\usepackage[top=1in, bottom=1in, left=1in, right=1in]{geometry} %Controls margins

\newcommand{\href}[1]{#1} % does nothing, but defines the command so the
\usepackage{ifthen}
\newboolean{ElectronicVersion}
\setboolean{ElectronicVersion}{true} % CHANGE THIS VALUE AS REQUIRED
\usepackage{amsmath,amssymb,amstext} % Lots of math symbols and environments
\usepackage[pdftex]{graphicx} % For including graphics N.B. pdftex graphics driver 
\usepackage{setspace}
\usepackage{subfig}

\usepackage{epstopdf} 

\ifthenelse{\boolean{ElectronicVersion}}{
    \usepackage[letterpaper=true,pdftex,bookmarks,pagebackref,
	plainpages=false, % Needed if Roman numbers in frontpages
        pdfpagelabels=true, % Adds page number as label in Acrobat's page count 
        pdftitle={Guaranteed clustering and biclustering via SDP}, % CHANGE THIS TEXT!
        pdfauthor=Brendan Ames	 % CHANGE THIS TEXT!
        ]{hyperref}}{}

\let\origdoublepage\cleardoublepage
\newcommand{\clearemptydoublepage}{%
  \clearpage{\pagestyle{empty}\origdoublepage}}
\let\cleardoublepage\clearemptydoublepage

\newtheorem{theorem}{Theorem}[section]

\newtheorem{cor}{Corollary}[section]

\newtheorem{lemma}{Lemma}[section]

\newcommand{\A}{\mathcal{A}}

\newcommand{\R}{\mathbf{R}}

\newcommand{\E}{\mathbf{E}}

\newcommand{\Diag}{\mathrm{Diag}\,}

\newcommand{\Null}{\mathrm{Null}\,}
\newcommand{\rank}{\mathrm{rank}\,}
\newcommand{\spn}{\mathrm{span}\,}
\newcommand{\st}{\mathrm{s.t.}\;}
\newcommand{\tr}{\mathrm{Tr}\,}

\renewcommand{\b}{\mathbf{b}}

\newcommand{\e}{\mathbf{e}}

\renewcommand{\u}{\mathbf{u}}
\newcommand{\bv}{\mathbf{v}}
\newcommand{\x}{\mathbf{x}}
\newcommand{\y}{\mathbf{y}}
\newcommand{\z}{\mathbf{z}}
\newcommand{\w}{\mathbf{w}}

\newcommand{\vect}[1] {\ensuremath{\left(\begin{array}{c} #1 \end{array} \right)}} % column vector.
\newcommand{\mat}[1] {\ensuremath{ \left(\begin{array} #1 \end{array} \right)}} %matrix: 1 includes dimensions, alignment, and contents.
\newcommand{\branchdef}[1] {\ensuremath{ \left\{\begin{array}{ll} #1 \end{array} \right. }} %conditional def of variable..
\newcommand{\rbra}[1]{\ensuremath{\left( #1 \right)}} % (1).
\newcommand{\bbra}[1]{\ensuremath{\left\{ #1 \right\}}} % (1).
\newcommand{\floor}[1]{\ensuremath{\left\lfloor #1 \right\rfloor}}
\newcommand{\ceil}[1]{\ensuremath{\left\lceil #1 \right\rceil}}

\newcommand{\iid}{independent identically distributed (i.i.d.) }
\newcommand{\pteto}[1]{probability tending exponentially to $1$ as $#1 \ra \infty$}

\newcommand{\kdc}{$k$-disjoint-clique }
\newcommand{\kdcs}{$k$-disjoint-clique subgraph }
\newcommand{\kdb}{$k$-disjoint-biclique }
\newcommand{\kdbs}{$k$-disjoint-biclique subgraph }

\newcommand{\npm}{\mathit{npm}}

\newcommand{\ra}{\rightarrow}

\newcommand{\qed}{\hfill\rule{2.1mm}{2.1mm}}

\numberwithin{equation}{section}
\numberwithin{figure}{section}
\numberwithin{table}{section}
\title{Guaranteed clustering and biclustering via semidefinite programming}
\author{Brendan P.W.~Ames
\thanks{Department of Computing + Mathematical Sciences,
California Institute of Technology,
1200 E. California Blvd., Mail Code 305-16, Pasadena, CA, 91125,  bpames@caltech.edu} }

\begin{document}

%----------------------------------------------------------------------
% FRONT MATERIAL
%----------------------------------------------------------------------

\maketitle

\begin{abstract}
Identifying clusters of similar objects in data plays a significant role in a wide range of applications. 
As a model problem for clustering, we consider the densest $k$-disjoint-clique problem, whose goal is to identify the collection of $k$
 disjoint cliques of a given weighted complete graph
 maximizing the sum of the densities of the complete subgraphs induced by these cliques.
In this paper, we establish conditions ensuring exact recovery of the densest $k$ cliques of a given graph
from the optimal solution of a particular semidefinite program.
In particular, the semidefinite relaxation is exact for input graphs corresponding to data consisting of $k$ large, distinct clusters and a smaller number of outliers.

This approach also yields a semidefinite relaxation with similar recovery guarantees for the biclustering problem. Given a set of objects and a set of features exhibited by these objects, biclustering seeks to
 simultaneously group the objects and features according to their expression levels. 
 This problem may be posed as that of partitioning the nodes of a weighted bipartite complete graph such that the sum of the densities of the resulting
 bipartite complete subgraphs is maximized.
As in our analysis of the densest $k$-disjoint-clique problem,
we show that the correct partition of the objects and features can be recovered from the optimal solution
of a semidefinite program in the case that the given data consists of several disjoint sets of objects exhibiting similar features.
Empirical evidence from numerical experiments supporting these theoretical guarantees is also provided.
\end{abstract}
%\doublespace

%----------------------------------------------------------------------
% MAIN BODY
%----------------------------------------------------------------------
%======================================================================
\section{Introduction}
%======================================================================

% Clustering def
The goal of {\it clustering} is  to partition a given data set into groups of similar objects, called {\it clusters}.
Clustering is a fundamental problem in statistics and machine learning and plays a significant role in a wide range of applications,
including information retrieval, pattern recognition, computational biology, and image processing.
The complexity of finding an optimal clustering depends significantly on the measure of fitness of a proposed partition,
but most interesting models for clustering are posed as an intractable combinatorial problem. 
For this reason, heuristics are used to cluster data in most practical applications.
Unfortunately, although much empirical evidence exists for the usefulness of these heuristics, few theoretical guarantees ensuring the quality of the obtained partition are known,
even for data containing well separated clusters.
For a recent survey of clustering techniques and heuristics, see \cite{berkhin2006survey}.
% Goals.
In this paper, we establish conditions ensuring that the optimal solution of
a particular convex optimization problem yields a correct clustering under certain assumptions on the input data set.
%The conditions guaranteeing exact recovery essentially require the data to contain a moderate number of
%sufficiently distinct clusters.

% Graph partitioning.
Our approach to clustering is based on partitioning the similarity graph of a given set of data.
Given a data set $S$ and measure of similarity between any two objects, the {\it similarity graph} $G_S$ is the
weighted complete graph with nodes corresponding to the objects in the data set and each edge $ij$ having weight
equal to the level of similarity between objects $i$ and $j$.
For this representation of data, clustering the data set $S$ is equivalent to partitioning the nodes of $G_S$ into disjoint cliques
such that edges connecting any two nodes in the same clique have significantly higher weight than those
between different cliques.
Therefore, a clustering of the data may be obtained by identifying dense, in the sense of having large average edge weight,
subgraphs of $G_S$.

% Densest k partition.
We consider the densest $k$-partition problem as a model problem for clustering.
Given a weighted complete graph $K = (V,E,W)$ and integer $k \in \{1,\dots, |V|\}$, the
{\it densest $k$-partition problem} aims to identify the partition of $V$ into $k$ disjoint sets
such that the sum of the average edge weights of the complete subgraphs induced by these cliques is maximized.
% Hardness.
Unfortunately, the densest $k$-partition problem is NP-hard, since it contains the minimum sum of squared Euclidean distance problem,
known to be NP-hard \cite{aloise2009np}, as a special case.
% Relaxation to densest kdc problem.
In Section \ref{sec: C results}, we consider the related problem of finding the set of $k$ disjoint complete subgraphs maximizing
the sum of their densities.
We model this problem as a quadratic program with combinatorial constraints and relax to a semidefinite program
using matrix lifting.
% State of the art.
This relaxation approach is similar to that employed in several recent papers \cite{peng2007approximating, singh2010ensemble, fan2012multi}, although
we consider a different model problem for clustering and establish stronger recovery properties.
% Recovery guarantees.
We show that the optimal solution of this semidefinite relaxation coincides with that of the original combinatorial problem
for certain program inputs.  In particular, we show that the set of input graphs for which the relaxation is exact includes the set of graphs
with edge weights concentrated on a particular collection of disjoint subgraphs, and
provide a general formula for the clique sizes and number of cliques that may be recovered.

% Extension to biclustering.
In Section \ref{sec: B results}, we establish similar results for the biclustering problem.
Given a set of objects and features, {\it biclustering}, also known as {\it co-clustering}, aims to simultaneously group the
objects and features according to their expression levels.
That is, we would like to partition the objects and features into groups of objects and features, called {\it biclusters}, such that
objects strongly exhibit features within their bicluster relative to the features within the other biclusters.
Hence, biclustering differs from clustering in the sense that it does not aim to obtain groups of similar objects, but instead
seeks groups of objects similar with respect to a particular subset of features.
Applications of biclustering include identifying subsets of genes exhibiting similar expression patterns across subsets of experimental
conditions in analysis of gene expression data, grouping documents by topics in document clustering, and grouping customers 
according to their preferences in collaborative filtering and recommender systems.
For an overview of the biclustering problem, see \cite{busygin2008biclustering, fan2010recent}.

% Our form/results.
As a model problem for biclustering, we consider the problem of partitioning a bipartite graph into dense disjoint subgraphs.
If the given bipartite graph has vertex sets corresponding to sets of objects and features with
edges indicating expression level of each feature by each object,
each dense subgraph will correspond to a bicluster of objects strongly exhibiting the contained features.
Given a weighted bipartite complete graph $K = ((U,V), E, W)$ and integer $k \in \{1,\dots, \min\{|U|, |V| \} \}$,
we seek the set of $k$ disjoint bipartite complete subgraphs with sum of their densities maximized.
We establish that this problem may be relaxed as a semidefinite program and show that, for certain program instances,
the correct partition of $K$ can be recovered from the optimal solution of this relaxation.
In particular, this relaxation is exact in the special case that the edge weights of the input graph
are concentrated on some set of  disjoint bipartite subgraphs.
When the input graph arises from a given data set, the relaxation is exact when the underlying data set consists of several disjoint
sets strongly exhibiting nonoverlapping sets of features.

% Relate to state of the art.
Our results build upon those of recent papers regarding clusterability of data.
These papers generally contain results of the following form:
if a data set is randomly sampled from a distribution of ``clusterable" data, then the correct partition of the data
can be obtained efficiently using some heuristic, such as the $k$-means algorithm or other iterative partitioning heuristics~\cite{Ostro, ackerman2009clusterability, shamir2002improved, bansal2004correlation}, 
spectral clustering \cite{ng2002spectral, kannan2004clusterings,  rohe2011spectral, balakrishnan25noise},
or convex optimization \cite{AV2, mathieu2010correlation, jalali2011clustering, oymak2011finding}.
Recent papers by Kolar et al. \cite{kolar2011minimax},  Rohe and Yu \cite{rohe2012co}, and Flynn and Perry \cite{flynn2012consistent} establish analogous recovery guarantees for biclustering;
the latter two of these papers appeared shortly after the initial preprint release of this paper.
%Our results are of an identical form.
Our results are of a similar form.
If the underlying data set consists of  several sufficiently distinct  clusters or biclusters, then the correct
partition of the data can be recovered from the optimal solution of our relaxations.
We model this ideal case for clustering using random edge weight matrices constructed so that weight is, in expectation,
concentrated heavily on the edges of a few disjoint subgraphs.
We will  establish that this random model for clustered data contains those previously considered in the literature and, in this
sense, our results are a generalization of these earlier theoretical guarantees.

% Relate to sparse optimization.
More generally, our results follow in the spirit of, and borrow techniques from, recent work regarding sparse optimization and, in particular, the nuclear norm relaxation for rank minimization.
The goal of {\it matrix rank minimization} is to find a solution of minimum rank of a given linear system, i.e., to find the optimal solution $X^* \in \R^{m\times n}$ of the 
optimization problem $\min\{\rank X: \A(X) = \b\}$ for given linear operator $\A: \R^{m\times n} \ra \R^p$ and vector $\b \in \R^p$.
Although this problem is well-known to be NP-hard, 
several recent papers (\cite{RFP, Candes-Recht:2008, Ames-Vavasis, gross2009recovering, recht2009simpler, candes2010tight, recht2010null, oymak2011simplified}, among others)
have established that, under certain assumptions on $\A$ and $\b$, the minimum rank solution is equal to the optimal solution of the convex relaxation obtained
by replacing $\rank X$ with the sum of the singular values of $X$, the nuclear norm $\|X\|_*$.
This relaxation may be thought of as a matrix analogue of the $\ell_1$ norm relaxation for the cardinality minimization problem, and these results generalize similar
recovery guarantees for  compressed sensing (see \cite{candes2005decoding, candes2006stable, donoho2006}).
For example, the nuclear norm relaxation is  exact with high probability if $\A$ is a  random linear transform with matrix representation
having i.i.d. Gaussian or Bernoulli entries and $\b = \A(X_0)$ is the image of a sufficiently low rank matrix $X_0$ under $\A$.
We prove analogous results for an instance of rank constrained optimization.
To identify the densest $k$ complete subgraphs of a given graph, we seek  a rank-$k$ matrix  $X$ maximizing some linear function of $X$, depending only on the edge weights $W$ of the input graph,
subject to linear constraints.
We will see that the optimal rank-$k$ solution is equal to that obtained
by relaxing the rank constraint to the corresponding nuclear norm constraint if the matrix $W$ is randomly sampled from a probability distribution satisfying certain assumptions.

%======================================================================
\section{A semidefinite relaxation of the densest \kdc problem}
\label{sec: C results}
%======================================================================
Given a graph $G = (V,E)$, a {\it clique} of $G$ is a pairwise adjacent subset of $V$. That is, $C \subseteq V$ is a clique of $G$ if $ij \in E$ for every pair of nodes $i,j \in C$.
Let $K_N= (V,E, W)$ be a complete graph with vertex set $V = \{1,2,\dots, N\}$ and nonnegative edge weights $W_{ij} \in [0,1]$ for all $i,j \in V$.
A {\it $k$-disjoint-clique subgraph} of $K_N$ is a subgraph of $K_N$ consisting of $k$ disjoint complete subgraphs, i.e., the vertex sets of each of these subgraphs
is a {\it clique}.
For any subgraph $H$ of $K_N$, the density of $H$, denoted $d_H$, is the average edge weight incident at a vertex in $H$:
$$
	d_H = \sum_{ij \in E(H)} \frac{W_{ij} }{|V(H)|}.
$$
The {\it densest \kdc problem} concerns choosing a \kdcs of $K_N$ such that the sum of the densities of the subgraphs induced by the cliques is
maximized.
Given a \kdcs with vertex set composed of cliques $C_1, \dots, C_k$, the sum of the densities of the subgraphs induced by the cliques is equal to
\begin{equation}	\label{eq: density def}
	\sum_{i=1}^k d_{G(C_i)} = \sum_{i=1}^k \frac{\bv_i^T W \bv_i}{ \bv_i^T \bv_i },
\end{equation}
where $\bv_i$ is the characteristic vector of $C_i$.
% Equivalence with k-means problem.
In the special case that $C_1,\dots, C_k$ defines a partition of $V$ and $W_{ij} = 1 - \|\x^{(i)} - \x^{(j)}\|^2$ for a given set of $N$ vectors $\{\x^{(1)}, \dots, \x^{(N)}\}$ in $\R^n$
with maximum distance between any two points at most one, we have
\begin{align*}
	\sum_{\ell=1}^k d_{G(C_\ell) } &= \sum_{\ell=1}^k  \frac{1}{|C_\ell|} \rbra{ \sum_{i \in C_{\ell}} \sum_{j \in C_{\ell}} (1 - (\x^{(i)} - \x^{(j)})^T (\x^{(i)} - \x^{(j)}))  }  \\
		&= \sum_{\ell=1}^k \left(|C_\ell| - 2 \left( \sum_{i\in C_\ell} \|\x^{(i)}\|^2 - \sum_{i\in C_\ell } \sum_{j\in C_\ell} (\x^{(i)})^T \x^{(j)}  \right) \right) \\
		&= N -  2\sum_{\ell=1}^k  \sum_{i\in C_\ell}\|\x^{(i)} - \mathbf{c}^{(\ell)} \|^2,
\end{align*}
where $\mathbf{c}^{(\ell)} = \sum_{i\in C_\ell} \x^{(i)}/|C_\ell|$ is the center of the vectors assigned to $C_\ell$ for all $\ell=1, \dots, k$,
since $\sum_{\ell=1}^k |C_\ell| = N$ for this choice of $W$.
Here, and in the rest of the note,
$\|\x\| = \sqrt{\x^T \x}$ denotes the $\ell_2$ norm in $\R^n$.
For this choice of $W$, the {\it densest $k$-partition problem}, i.e., finding a partition $C_1,\dots, C_k$ of $V$ such that the sum of densities of the subgraphs
induced by $C_1,\dots, C_k$ is maximized,
is equivalent to finding the partition of $V$ such that the sum of the squared Euclidean distances
\begin{equation}	\label{eq: f def}
	f(\{\x^{(1)},\dots, \x^{(n)}\}, \{ C_1,\dots, C_k\} ) = \sum_{\ell=1}^k  \sum_{i\in C_\ell}\|\x^{(i)} - \mathbf{c}^{(\ell)} \|^2
\end{equation}
from each vector $\x^{(i)}$ to its assigned cluster center is minimized.
Unfortunately, minimizing  $f$ over all potential partitions of $V$ is NP-hard and, thus, so is the densest $k$-partition problem (see \cite{peng2007approximating}).
It should be noted that the complexity of the densest \kdcs problem is unknown, although the problem of minimizing $f$ over all \kdc subgraphs has the trivial solution of
assigning exactly one point to each cluster and setting all other points to be outliers.

% Def of NPMS.
If we let $X$ be the $N\times k$ matrix with $i$th column equal to $\bv_i/\|\bv_i\|$, we have
$
	\sum_{i=1}^k  d_{G(C_i)} = \tr(X^T W X).
$
We call such a matrix $X$ a {normalized $k$-partition matrix}.
That is, $X$ is a {\it normalized k-partition matrix} if the columns of $X$ are the normalized characteristic vectors of $k$ disjoint subsets of $V$.
We denote by $\npm(V,k)$ the set of all normalized $k$-partition matrices of $V$.
We should note that the term normalized $k$-partition matrix is a slight misnomer;
the columns of $X \in \npm(V,k)$ do not necessarily define a partition of $V$ into $k$ disjoint sets
but do define a partition of $V$ into the $k+1$ disjoint sets given by the columns
$X(:,1), \dots, X(:,k)$ of $X$ and their complement.
Using this notation, the densest \kdc problem may be formulated as the quadratic program
\begin{equation} \label{eq: KDC comb QP}
	\max \{ \tr(X^T W X) : X \in \npm(V,k) \}.
\end{equation}
Unfortunately, quadratic programs with combinatorial constraints are NP-hard in general.

The quadratic program~\eqref{eq: KDC comb QP} may be relaxed to a rank constrained semidefinite program using matrix lifting.
We replace each column $\x_i$ of $X$ with a rank-one semidefinite variable $\x_i \x_i^T$ to obtain the new decision variable
%\begin{equation}	\label{eq: lifted X}
$
	\tilde X = \sum_{i=1}^k \x_i \x_i^T.
$	
%\end{equation}
The new variable $\tilde X$ has both rank and trace exactly equal to $k$ since the summands $\x_i\x_i^T$ are orthogonal rank-one matrices, each with nonzero eigenvalue equal to 1.
Moreover, since $[\x_i]_{j} = 1/\sqrt{r_i}$ for all $j\in C_i$ and all remaining components equal to $0$ where $r_i$ is equal to the number of nonzero entries of $\x_i$  , we have
$$
	\tilde X \e = \sum_{i=1}^k \x_i (\x_i^T \e) = \sum_{i=1}^k \sqrt{r_i} \x_i.
$$
Thus, the matrix $\tilde X$
has row sum equal to one for each vertex in the subgraph of $K_N$ defined by the columns of $X$ and zero otherwise.
Therefore, we may relax \eqref{eq: KDC comb QP} as the rank constrained semidefinite program
\begin{equation}	\label{E: w rank relax}
	\max \big\{\tr (WX) : X \e \le \e, \rank X = k, \tr X = k, X \ge 0, X \in \Sigma^N_+ \big\}
\end{equation}
Here $\Sigma^N_+$ denotes the cone of $N\times N$ symmetric positive semidefinite matrices and $\e$ denotes the all-ones vector in $R^{N}$.
The nonconvex program \eqref{E: w rank relax} may be relaxed further to a semidefinite program by ignoring the nonconvex
constraint $\rank(X) = k$:
\begin{equation}	\label{E: w trace relax}
	\max \big\{\tr( W X ) : X \e \le \e, \tr X = k, X \ge 0, X \in \Sigma^N_+ \big\}.
\end{equation}
Note that a \kdcs with vertex set composed of disjoint cliques $C_1,$ $\dots,$ $C_k$ defines a feasible solution of \eqref{E: w trace relax} with rank exactly equal to $k$
and objective value equal to \eqref{eq: density def}
by
\begin{equation}	\label{E: proposed sol}
	X^* = \sum_{i=1}^k \frac{\bv_i \bv_i^T}{\bv_i^T \bv_i},
\end{equation}
where $\bv_i$ is the characteristic vector of $C_i$ for all $i =1,\dots, k$.
This feasible solution is exactly the lifted solution corresponding to the cliques $\{C_1, \dots, C_k\}$.
This relaxation approach mirrors that for the planted $k$-disjoint-clique problem considered in \cite{AV2}.
In \cite{AV2}, entrywise nonnegativity constraints can be ignored for the sake  of computational efficiency due to explicit constraints
forcing all entries of a feasible solution corresponding to unadjacent nodes to be equal to 0.
Due to the lack of such constraints in \eqref{E: w trace relax}, the nonnegativity constraints are required to ensure that the optimal solution
of \eqref{E: w trace relax} is unique if the input data is sufficiently clusterable.
Indeed, suppose that 
$$
	W = \mat{ {cc}\e\e^T & 0 \\ 0 & \e\e^T },
$$
where $\e$ is the all-ones vector in $\R^n$.
Then both 
$$
	\frac{1}{n} \mat{{cc} \e\e^T & 0 \\ 0 &\e\e^T}  \mbox{\;\; and \;\; } \frac{1}{n} \mat{{cc}\e\e^T & -\e\e^T \\ -\e\e^T & \e\e^T }
$$
are positive semidefinite, have trace equal to $2$, row sums bounded above by $1$, and have objective value equal to $2n$.
Therefore, the nonnegativity constraints are necessary to distinguish between these two solutions.
We should also point out that the constraints of \eqref{E: w trace relax} are similar to those of the
semidefinite relaxation used to approximate the minimum sum of squared Euclidean distance partition  by Peng and Wei
in \cite{peng2007approximating}, although with different derivation.

% Relation to NNM.
The relaxation \eqref{E: w trace relax} may be thought of as a nuclear norm relaxation of \eqref{E: w rank relax}.
Indeed, since the eigenvalues and singular values of a positive semidefinite matrix are identical, every feasible solution $X$
satisfies
$
	\tr(X) = \sum_{i=1}^N \sigma_i(X) = \|X\|_*.
$
Moreover, since every feasible solution $X$ is symmetric and has row sums at most $1$, we have 
$
	\|X\|_1 = \|X\|_\infty \le 1
$
for every feasible $X$.
Here $\|\cdot\|_1$, $\|\cdot\|$, and $\|\cdot\|_\infty$ denote the matrix norms on $\R^{N\times N}$ induced by the $\ell_1$, $\ell_2$, and $\ell_\infty$ norms on $\R^N$
respectively.
This implies that every feasible $X$ satisfies $\|X\| \le 1$ since $\|X\| \le \sqrt{\|X\|_1 \|X\|_\infty}$ (see  \cite[Corollary~2.3.2]{GV}).
Since $\|X\|_*$ is the convex envelope of $\rank(X)$ on the set $\{X : \|X\|\le 1\}$ (see, for example, \cite[Theorem 2.2]{RFP}),
the semidefinite program \eqref{E: w trace relax} is exactly the relaxation of \eqref{E: w rank relax} obtained by ignoring the 
rank constraint and only constraining the nuclear norm of a feasible solution.
Many recent results %\cite{RFP, Candes-Recht:2008, Ames-Vavasis, gross2009recovering, recht2009simpler, candes2010tight, recht2010null, oymak2011simplified}
have shown that the minimum rank solution of a set of linear equations $\A(X) = \b$
is equal to the minimum nuclear norm solution, under certain assumption on the linear operator $\A$.
We  would like to prove analogous results for the relaxation \eqref{E: w trace relax}.
That is, we would like to identify conditions on the input graph that guarantee recovery of the densest \kdcs by solving \eqref{E: w trace relax}.

% Planted cluster model.
Ideally, a clustering heuristic should be able to correctly identify the clusters in data that is known {\it a priori} to be clusterable.
In our graph theoretic model, this case corresponds to a graph $G_S=(V,E,W)$ admitting a \kdcs with very high weights on edges connecting nodes within the cliques
and relatively low weights on edges between different cliques.
We focus our attention on input instances for the densest \kdc problem that are constructed to possess this structure.
Let $K^*$ be a \kdcs of $K_N$ with vertex set composed of disjoint cliques $C_1, C_2, \dots, C_k$.
We consider random symmetric matrices $W \in \Sigma^N$ with entries sampled independently from one of two distributions $\Omega_1$, $\Omega_2$ as follows:
\begin{itemize}
	\item		
		For each $q=1,\dots, k$,
		the entries of each diagonal block $W_{C_q,C_q}$ are independently sampled from a probability distribution $\Omega_1$ satisfying
		%\begin{equation} \label{E: kdc omega1}
			$\E[W_{ij}] = \E[W_{ji}] = \alpha$
		%\end{equation}
		and $W_{ij} \in [0,1]$ for all $i,j \in C_q$.		
	\item
		All remaining entries of $W$ are independently sampled from a probability distribution $\Omega_2$ satisfying		
		%\begin{equation} \label{E: kdc omega2}
			$\E[ W_{ij} ] = \E[W_{ji} ] = \beta$
		%\end{equation}
		and $W_{ij} \in [0,1]$ for all $(i,j) \in (V\times V )\setminus \cup_{q=1}^k (C_q\times C_q)$.
\end{itemize}
That is,  we  sample the random variable $W_{ij}$ from the probability distribution $\Omega_1$ with
mean $\alpha$ if the nodes $i,j$ are in the same planted clique; otherwise, we sample $W_{ij}$ from the distribution $\Omega_2$ with mean $\beta$.
We say that such random matrices $W$ are sampled from the {\it planted cluster model}.
We should note the planted cluster model is a generalization of the planted \kdcs model considered in \cite{AV2}, as well as the stochastic block/probabilistic cluster model considered in
\cite{jalali2011clustering, oymak2011finding, rohe2011spectral}. Indeed, the stochastic block model is generated by independently adding edges within planted dense subgraphs with probability $p$ and 
independently adding edges between cliques with probability $q$ for some $p > q$. The planted \kdcs model is simply the stochastic block model in the special case that $p = 1$.
Therefore, choosing $\Omega_1$ and $\Omega_2$ to be  Bernoulli distributions with probabilities of success $p$ and $q$, respectively,  yields $W$ sampled from the stochastic block model.		

The following theorem describes which partitions $\{C_1, C_2, \dots, C_{k+1}\}$ of $V$ yield random symmetric matrices $W$ drawn from the planted cluster model
such that the corresponding planted \kdcs $K$ is the densest \kdcs and can be found with high probability by
solving \eqref{E: w trace relax}.

%----------------------------------------------------------------------------------
% Main theorem.
%----------------------------------------------------------------------------------
\begin{theorem} \label{T: weighted kdc results}
	Suppose that the vertex sets $C_1, \dots, C_k$ define a $k$-disjoint-clique subgraph $K^*$ of the complete graph $K_N = (V, E)$
	on $N$ vertices and let $C_{k+1} := V \setminus (\cup^k_{i=1} C_i).$
	Let $r_i := |C_i|$ for all $i=1,\dots, k+1$, and let $\hat r = \min_{i=1,\dots, k} r_i$.
	Let $W \in \Sigma^N$ be a random symmetric matrix sampled from the planted cluster model according to distributions $\Omega_1$ and $\Omega_2$
	with means $\alpha$ and $\beta$, respectively, satisfying 
	\begin{equation} \label{A: alpha beta ratio}
		\gamma = \gamma(\alpha, \beta, r)  := \alpha ( 1 + \delta_{0, r_{k+1}})  - 2 \beta > 0,
	\end{equation}
	where $\delta_{i,j}$ is the Kronecker delta function defined by $\delta_{i,j} = 1$ if $i=j$ and $0$ otherwise.
	Let $X^*$ be the feasible solution for \eqref{E: w trace relax} corresponding to $C_1, \dots, C_k$ defined
	by \eqref{E: proposed sol}.		
	Then there exist scalars $c_1, c_2, c_3 > 0$ such that if
	\begin{equation}	\label{A: wkdc guarantee bound}
		c_1  \sqrt{N} + c_2 \sqrt{k r_{k+1}} + c_3 r_{k+1} \le \gamma\hat r
	\end{equation}
	then $X^*$ is the unique optimal solution for \eqref{E: w trace relax}, and $K^*$  is the 
	unique maximum density $k$-disjoint-clique
	subgraph of $K_N$ corresponding to $W$ with \pteto{\hat r}.
%	Moreover, $X^*$ is the unique optimal solution of \eqref{E: w trace relax} and
%	$K^*$ is the unique maximum mean weight $k$-disjoint-clique
%	subgraph of $K_N$  if
%	\begin{equation} \label{A: block weights}
%		r_s \e^T W_{C_q, C_q} \e > r_q \e^T W_{C_q, C_s} \e
%	\end{equation}
%	for all $q, s \in \{1,\dots,k\}$ such that $q\neq s$.
\end{theorem}

%----------------------------------------------------------------------------------
% partition examples.
%----------------------------------------------------------------------------------

Note that the condition \eqref{A: alpha beta ratio} implies that $\alpha > \beta$ if $r_{k+1} = 0$ and $\alpha > 2 \beta$ otherwise. That is, if  $\{C_1,\dots, C_k\}$
defines a  partition of $V$ then the restriction that $\alpha > 2 \beta$ can be relaxed to $\alpha > \beta$.
On the other hand,
the condition \eqref{A: wkdc guarantee bound}
cannot be satisfied unless $N = O(\hat r^2)$ and $r_{k+1} = O(\hat r)$.
We now provide a few examples of $r_1, \dots, r_k$ satisfying the hypothesis of Theorem~\ref{T: weighted kdc results}.

\begin{itemize}
\item 
	Suppose that we have $k$ cliques $C_1, \dots, C_k$ of size $r_1 = r_2 = \dots = r_k = N^{\epsilon}$.
	Then \eqref{A: wkdc guarantee bound} implies that we may recover the $k$-disjoint-clique subgraph corresponding
	to $C_1, \dots, C_k$ if  $N^{\epsilon} \ge \Omega(N^{1/2})$.
	Since the cliques $C_1,\dots, C_k$ are disjoint 
	and contain $\Omega(N)$ nodes, we must have $\epsilon \ge 1/2$. Therefore, our heuristic may recover $O(N^{1/2})$ planted cliques of size $N^{1/2}$.
\item 
	On the other hand, we may have cliques of different sizes.
	For example, suppose that we wish to recover $k_1$ cliques of size $N^{3/4}$ and $k_2$ smaller 
	cliques of size $N^{1/2}$. Then the right-hand side of \eqref{A: wkdc guarantee bound} must be at least
	$
		\Omega\big( \sqrt{N} + \sqrt{(k_1 + k_2) r_{k+1} } + r_{k+1}\big).
	$
	Therefore, we may recover the planted cliques provided that $k_1 = O(N^{1/4})$, $k_2 = O(N^{1/2})$, and $r_{k+1} = O(N^{1/2})$.
\end{itemize}

Although we consider a more general model for clustered data, our recovery guarantee agrees (up to constants) with those existing in the literature.
In particular, the bound on the minimum size of the planted clique recoverable by the relaxation \eqref{E: w trace relax}, $\hat r = \Omega(N^{1/2})$, provided by Theorem~\ref{T: weighted kdc results}
 matches that given in \cite{jalali2011clustering, oymak2011finding}.
However, among the existing recovery guarantees in the literature, few consider noise in the form of diversionary nodes.
As a consequence of our more general model, the relaxation~\eqref{E: w trace relax} is exact for input graphs containing up to $O(\hat r)$ noise nodes,  fewer than the bound, $O(\hat r^2)$, provided by \cite[Theorem~4.5]{AV2}.

%======================================================================
\section{A semidefinite relaxation of the densest \kdb problem}
\label{sec: B results}
%======================================================================
% Preliminaries/definitions.
Given a bipartite graph $G=((U,V),E)$, a pair of disjoint independent subsets $U' \subseteq U$, $V' \subseteq V$ is a {\it biclique} of
$G$ if the subgraph of $G$ induced by $(U',V')$ is complete bipartite. That is, $(U',V')$ is a biclique of $G$ if $uv \in E$ for all $u \in U', v\in V'$.
A {\it \kdb subgraph} of $G$ is a subgraph of $G$ with vertex set composed of $k$ disjoint bicliques of $G$.
Let $K_{M,N} = ((U,V),E, W)$ be a weighted complete bipartite graph with vertex sets $U = \{1,2,\dots, M\}$, $V = \{1,\dots, N\}$
with matrix of edge weights $W \in [0,1]^{U\times V}$.
We are interested in identifying the densest \kdb subgraph of $K_{M,N}$ with respect to $W$.
We define the density of a subgraph $H = (U',V',E')$ of $K_{M,N}$ to be the total edge weight incident at each vertex divided by the square root of the number of edges from $U'$ to $V'$:
\begin{equation}	\label{eq: b density def}
	d_H =  \frac{1}{  \sqrt{|E'|} }\sum_{u\in U', v \in V'} {W_{uv}}.
\end{equation}
Note that the density of $H$, as defined by \eqref{eq: b density def}, is not necessarily equal to the average edge weight incident at a vertex of $H$, since the square root of the number of edges
is not equal to the total number of vertices
if $|U'| \neq |V'|$ or $H$ is not complete.
The goal of the {\it densest \kdb problem} is to identify a set of $k$ disjoint bicliques of $K_{M,N}$
such that the sum of the densities of the complete subgraphs induced by these bicliques is maximized. That is, we want to find a set of $k$ disjoint bicliques, with characteristic vectors
$ (\u_1, \bv_1), \dots, (\u_k, \bv_k)$, maximizing the sum
% Combinatorial objective.
\begin{equation} \label{eq: WKDB objective}
	\sum_{i=1}^k \frac{\u_i^T W \bv_i}{\|\u_i \| \| \bv_i \|}.
\end{equation}
% Formulation as nonconvex QP.
As in our analysis of the densest \kdc problem, this problem may be posed as the nonconvex quadratic program
\begin{equation} 	\label{eq: bicluster QP}
	\max \{ \tr ( X^T W Y) :  X \in \npm(U), \; Y \in \npm(V) \}.
\end{equation}
% Relaxation to rank constrained SDP.
By letting $Z = (X^T, \, Y^T)^T (X^T, \, Y^T)$, we have $\tr(X^T W Y) = \frac{1}{2} \tr(\tilde W Z)$, where
$$
	\tilde W = \mat{{cc} 0 & W \\ W^T & 0 }.
$$
Using this change of variables, we relax to the rank constrained semidefinite program
\begin{equation}	\label{eq: KDB rank form}
\begin{array}{ll}
	\max \;\; & \frac{1}{2}\tr( \tilde W Z ) \\
	\st \;\;	& Z_{U,U} \e \le \e, \;\;\; Z_{V,V} \e \le \e, \\
			& \tr ( Z_{U,U} ) = k, \;\;\; \tr ( Z_{V,V} ) = k ,\\
			& \rank (Z_{U,U} ) = k, \;\; \rank ( Z_{V,V} ) = k, \\
			& Z \ge 0, \;\;\;  Z \in \Sigma^{M+N}_+,
\end{array}
\end{equation}
where $Z_{U,U}$ and $Z_{V,V}$ are the blocks of $Z$ with rows and columns indexed by $U$ and $V$ respectively.
% Relaxation to SDP.
Ignoring the nonconvex rank constraints  yields the semidefinite relaxation
\begin{equation}	\label{eq: WKDB relaxation} 
\begin{array}{ll}
	\max \;\; & \frac{1}{2}\tr( \tilde W Z ) \\
	\st \;\;	& Z_{U,U} \e \le \e, \;\;\; Z_{V,V} \e \le \e, \\
			& \tr ( Z_{U,U} ) = k, \;\;\; \tr ( Z_{V,V} ) = k ,\\			
			&  Z \ge 0, \;\;\;  Z \in \Sigma^{M+N}_+.
\end{array}
\end{equation}
As in our analysis of the densest \kdc problem, we would like to identify sets of program instances of the \kdb problem that may be solved using the semidefinite relaxation \eqref{eq: WKDB relaxation}.
As before, we consider input graphs where it is known {\it a priori} that a \kdbs with large edge weights, relative to the edges of its complement, exists.
We consider random program instances generated as follows. Let $G^*$ be a \kdbs of $K_{M,N}$ with vertex set composed of the disjoint bicliques $(U_1, V_1), \dots, (U_k, V_k)$.
We construct a random matrix $W \in \R^{M\times N}_+$ with entries sampled independently from one of two distributions $\Omega_1, \Omega_2$ as follows.
\begin{itemize}
	\item If $u \in U_i$, $v \in V_i$ for some $i\in \{1,\dots, k\}$, then we sample $W_{uv}$ from the distribution $\Omega_1$, with mean $\alpha$. If $u$ and $v$ are in different bicliques of $K^*$, then we
		sample $W_{uv}$ according to the probability distribution $\Omega_2$, with mean $\beta < \alpha$.
	\item
		The probability distributions $\Omega_1, \Omega_2$ are chosen such that $u \in U, v\in V$, $0 \le W_{uv} \le 1$.		
\end{itemize}
We say that such $W$ are sampled from the {\it planted bicluster model}. Note that  $G^*$  defines a feasible solution for \eqref{eq: WKDB relaxation} by
\begin{equation}	\label{eq: WKDB sol}
	Z^* = \sum_{i=1}^k \vect{ \u_i/\|\u_i\| \\ \bv_i / \|\bv_i \| } \vect{ \u_i/\|\u_i\| \\ \bv_i / \|\bv_i \| }^T,
\end{equation}
where $\u_i, \bv_i$ are the characteristic vectors of $U_i$ and $V_i$, respectively, for all $i=1,\dots,k$.
Moreover, $Z^*$ has objective value equal to \eqref{eq: WKDB objective}.
The following theorem describes which partitions $\{U_1,\dots,U_k\}$ and $\{V_1, \dots, V_k\}$ of $U$ and $V$ yield random matrices $W$ 
drawn from the planted bicluster model such that $Z^*$ is the unique optimal solution of the semidefinite relaxation \eqref{eq: WKDB relaxation}
and  $G^*$ is the unique densest \kdb subgraph.

\begin{theorem}	\label{thm: biclustering recovery guarantee}
	Suppose that the vertex sets $(U_1, V_1), \dots, (U_k, V_k)$ define a \kdbs $K^*$ of the complete bipartite graph $K_{M,N} = ((U,V), E)$.
	Let $U_{k+1} := U \setminus (\cup_{i=1}^k U_i)$ and $V_{k+1} := V \setminus (\cup_{i=1}^k V_i)$.
	Let  $m_i = |U_i|$ and $n_i = |V_i|$ for all $i = 1, \dots, k+1$ and $\hat n:= \min_{i=1,\dots,k} n_i$.	
	Let $Z^*$ be the feasible solution for \eqref{eq: WKDB relaxation} corresponding to $K^*$ given by \eqref{eq: WKDB sol}.
	Let $W \in \R^{M\times N}_+$ be a random matrix sampled from the planted bicluster model
	according to distributions $\Omega_1$ and $\Omega_2$ with means $\alpha, \beta$ satisfying 
	\begin{equation}	\label{eq: a twice b}
		\gamma = \gamma (\alpha, \beta, m, n) := \alpha (1 + \delta_{0,m_{k+1}}\delta_{0, n_{k+1} } ) - 2\beta > 0.
	\end{equation}	
%	such that $E[\Omega_1] = \alpha$ and $E[\Omega_2] = \beta$ for some $\alpha >  2\beta$.
	Suppose that there exist scalars $\{\tau_1,\dots, \tau_{k+1}\}$ such that
	$m_i = \tau_i^2 n_i$ for all $i \in \{1,\dots, k+1\}$ and
	\begin{equation}	\label{eq: abd assumption}
		\alpha \tau_i > \beta \tau_j
	\end{equation}
	for all $i,j \in \{1,\dots, k+1\}$.
	Then there exist scalars $c_1, c_2 > 0$ depending only on $\alpha, \beta,$ and $\{\tau_1,\dots, \tau_{k+1}\}$ such that if
%	\begin{equation}	\label{eq: max b size}
%		n_i \le c_1 (\alpha - \beta)^2  \hat n^2
%	\end{equation}
%	and
	\begin{equation}	\label{eq: guarantee ineq}
		c_1 \rbra{ \sqrt{k} +\sqrt{n_{k+1}} + 1} \sqrt{N}  + \beta \tau_{k+1} n_{k+1} \le c_2 \gamma \hat n
	\end{equation}
	then $Z^*$ is the unique optimal solution of \eqref{eq: WKDB relaxation} and
	$G^*$ is the unique maximum density \kdbs with respect to $W$ with probability tending exponentially to
	$1$ as $\hat n$ tends to $\infty$.	
\end{theorem}

% Example of situation where it works.
For example, Theorem~\ref{thm: biclustering recovery guarantee} implies that $O(N^{1/3})$ bicliques of size $\hat m = \hat n = N^{2/3}$ can be recovered
from a graph sampled from the planted bicluster model with up to $O(N^{1/3})$ diversionary nodes
by solving \eqref{eq: WKDB relaxation}.

%%======================================================================
\section{Proof of Theorem~\ref{T: weighted kdc results}}
\label{sec: kdc proof}
%%======================================================================
This section comprises a proof of Theorem~\ref{T: weighted kdc results}.
The proof of Theorem~\ref{thm: biclustering recovery guarantee} is essentially identical to that of Theorem~\ref{T: weighted kdc results},
although with some modifications made to accommodate the different relaxation and lack of symmetry
of the weight matrix $W;$
an outline of the proof of Theorem~\ref{thm: biclustering recovery guarantee} is given in Section~\ref{sec: kdb proof}.

%\subsection{Notation}
%\label{sec: b notation}
%\input{Source/b-notation.tex}

\subsection{Optimality Conditions}
\label{sec: w opt conds}
% KKT conditions
We begin with the following sufficient condition for the optimality of a feasible solution of \eqref{E: w trace relax}.

\begin{theorem}
	\label{T:  w KKT conditions}
	Let $X$ be feasible for \eqref{E: w trace relax} and suppose that 
	there exist some $\mu \in \R$, $\lambda \in \R^N_+$, $\eta \in \R^{N\times N}_{+}$ and $S \in \Sigma^N_+$ such that
	\begin{align}
		- W + \lambda \e^T + \e \lambda^T - \eta + \mu I &= S \label{E: w dual feas} \\
		\lambda^T (X\e - \e)  &= 0  \label{E: w CS rowsum} \\		
		\tr(X  \eta) &= 0 \label{E: w CS nonneg} \\
		\tr (X S)  &= 0. \label{E: w CS sdp}
	\end{align}
	Then $X$ is optimal for \eqref{E: w trace relax}.
\end{theorem}

%----------------------------------------------------------------------------------
% Strict feasibility.
%----------------------------------------------------------------------------------

Note that 
$
	X = (k/N - \epsilon ) I + \epsilon \e\e^T
$
is  a strictly feasible solution of \eqref{E: w trace relax} for sufficiently small $\epsilon > 0$.
%and choosing $\lambda = 0,$ $\eta =0,$ $\phi = 0$ and $\mu_1, \mu_2$ large enough that
%the left-hand side of \eqref{E: w dual feas} is positive definite shows that the dual of \eqref{E: w trace relax} is strictly feasible.
Thus, Slater's constraint qualification holds for \eqref{E: w trace relax}.
Therefore, a feasible solution $X$ is optimal for \eqref{E: w trace relax} if and only if it satisfies the Karush-Kuhn-Tucker conditions.
Theorem~\ref{T:  w KKT conditions} provides the necessary specialization to \eqref{E: w trace relax} of these necessary and sufficient conditions (see, for example, \cite[Section 5.5.3]{boydvdb} or \cite[Theorem~28.3]{Rockafellar}).
%Theorem~\ref{T:  w KKT conditions} is nothing more than the specialization of the optimality conditions 
%given by the Karush-Kuhn-Tucker theorem to the semidefinite program~\eqref{E: w trace relax}

%----------------------------------------------------------------------------------
% Statement of assumptions.
%----------------------------------------------------------------------------------

Let $K^*$ be a $k$-disjoint-clique subgraph of $K_N$ with vertex set composed of the disjoint cliques $C_1, \dots, C_k$ of sizes $r_1, \dots, r_k$  and let
$X^*$ be the corresponding feasible solution of \eqref{E: w trace relax} defined by \eqref{E: proposed sol}.
Let $C_{k+1} := V \setminus  (\cup^k_{i=1} C_i)$ and $r_{k+1} := N - \sum_{i=1}^k r_i$. 
Let $\hat r := \min_{i=1, \dots, k} r_i$.
Let $W \in \Sigma^N$ be a random symmetric matrix sampled from the planted cluster model according to $\Omega_1$ and $\Omega_2$ with means $\alpha$ and $\beta$.
To show that $X^*$ is optimal for \eqref{E: w trace relax},  we will construct multipliers
$ \mu \in \R$, $\lambda \in \R^N_+$, $\eta \in \R^{N\times N}_{+}$, and $S \in \Sigma^N_+$ satisfying
\eqref{E: w dual feas},  \eqref{E: w CS rowsum}, \eqref{E: w CS nonneg},  and \eqref{E: w CS sdp}.
Note that the gradient equation \eqref{E: w dual feas} provides 
an explicit formula for the multiplier $S$ for any choice of multipliers $\mu, \lambda, $ and $\eta$.

%----------------------------------------------------------------------------------
% Proof idea.
%----------------------------------------------------------------------------------

The proof of Theorem~\ref{T: weighted kdc results} uses techniques similar to those used in \cite{AV2}.
Specifically, the proof of Theorem~\ref{T: weighted kdc results} relies on constructing multipliers satisfying Theorem~\ref{T:  w KKT conditions}.
The multipliers $\lambda$ and $\eta$ will be constructed in blocks inherited from the block structure of the proposed solution $X^*$.
Again, once the multipliers $\mu,  \lambda, $ and $\eta$ are chosen, \eqref{E: w dual feas} provides an explicit formula for the
multiplier $S$.

The dual variables must be chosen so that the complementary slackness condition \eqref{E: w CS sdp} is satisfied.
The condition $\tr(X^*S) = 0$ is satisfied if and only if $X^*S = 0$, since both $X^*$ and $S$ are desired to be positive semidefinite (see \cite[Proposition 1.19]{Tuncel}).
Therefore, the multipliers must be chosen so that the left-hand side of \eqref{E: w dual feas} is orthogonal to the columns of $X^*$.
That is, we must choose the multipliers $\mu, \lambda, $ and $\eta$ such that $S$, as defined by \eqref{E: w dual feas},
has nullspace containing the columns of $X^*$.
By the special block structure of $X^*$, this is equivalent to requiring $S_{v, C_q}$ to sum to $0$
for all $q \in \{1,\dots, k\}$ and $v \in V.$
%Similarly, the notation $\bv(A)$ will refer to the vector in $\R^{|A|}$ with entries corresponding to those of $\bv$ indexed by the set $A$.

The gradient equation~\eqref{E: w dual feas}, coupled with the requirement that the columns of $X^*$ reside in the nullspace of $S$,
provides an explicit formula for the multiplier $\lambda$.
Moreover, the complementary slackness condition~\eqref{E: w CS nonneg} implies that all diagonal blocks $\eta_{C_q, C_q}$, $q=1,\dots, k$,
are equal to $0$.
To construct the remaining multipliers, we parametrize the remaining blocks of $S$ using the vectors $\y^{q,s}$ and $\z^{q,s}$ for all $q \neq s$.
These vectors are chosen to be the solutions of the system of linear equations defined by $SX^* = X^*S = 0$.
As in \cite{AV2},
we will show that this system is a perturbation of a linear system with known solution and will use this known solution to obtain estimates of $\y^{q,s}$ and $\z^{q,s}$.

Once the multipliers are chosen, we must establish dual feasibility to prove that $X^*$ is optimal for \eqref{E: w trace relax}.
In particular, we must show that $\lambda$ and $\eta$ are nonnegative and $S$ is positive semidefinite.
To establish nonnegativity of $\lambda$ and $\eta$, we will show that these multipliers are strictly positive in expectation
and close to their respective  means with extremely high probability.
To establish that $S$ is positive semidefinite, we will show that the diagonal blocks of $S$ dominate the off diagonal blocks with high probability.

%
%\subsection{Bounds on the norms of random matrices and sums of random variables}
%\label{sec: background}
%\input{background.tex}

\subsection{Choice of the multipliers and a sufficient condition for uniqueness and optimality}

\label{sec: c multipliers}
%----------------------------------------------------------------------------------
% Block structure of multipliers
%----------------------------------------------------------------------------------
We construct the multipliers $\lambda, \eta$, and $S$  in blocks indexed by the vertex sets $C_1, \dots, C_{k+1}$.
The complementary slackness condition \eqref{E: w CS sdp} implies that the columns of $X$ are in the nullspace of $S$
since $\tr(XS) = 0$ if and only if $XS = 0$ for all positive semidefinite $X,S$.
Since $X^*_{C_q, C_q}$ is a multiple of the all-ones matrix $\e\e^T$ for each $q=1,\dots, k$,
and all other entries of $X^*$ are equal to $0$, \eqref{E: w CS sdp} implies that the block $S_{C_q, C_s}$
must have row and column sums equal to $0$ for all $q,s \in \{1,\dots, k\}$.
% Eta
Moreover, since all entries of $X^*_{C_q, C_q}$ are nonzero, $\eta_{C_q,C_q} = 0$ for all $q = 1,\dots, k$ by \eqref{E: w CS nonneg}.

% Lambda.
To compute an explicit formula for $\lambda$, note that
the condition $S_{C_q, C_q} \e = 0$ is satisfied if 
\begin{equation} \label{eq: lambda 1}
	0 = S_{C_q, C_q} \e = \mu \e  + r_q \lambda_{C_q} + (\lambda_{C_q}^T \e) \e - W_{C_q, C_q} \e
\end{equation}
for all $q=1,\dots,k$.
Rearranging \eqref{eq: lambda 1} shows that
$\lambda_{C_q}$ is the solution to the system
\begin{equation} \label{eq: lambda 2}
	(r_q I + \e\e^T) \lambda_{C_q} = W_{C_q, C_q} \e - \mu \e
\end{equation}
for all $q =1,\dots, k$.
We will use the Sherman-Morrision-Woodbury formula (see, for example, \cite[Equation (2.1.4)]{GV}), stated in the following lemma, to obtain the desired formula for $\lambda$.

% SMW formula
\begin{lemma} \label{lem: SMW}
	Let $A \in \Sigma^{n\times n}$ be nonsingular and $\u, \bv \in \R^{n}$ be such that $1 + \bv^T A^{-1} \u \neq 0$.
	Then
%	\begin{equation} \label{eq: SMW general formula}
%		(A + UV^T)^{-1} = A^{-1} - A^{-1}U ( I + V^T A^{-1} U)^{-1}V^T A^{-1}.
%	\end{equation}
%	Moreover,
%	we have
	\begin{equation} \label{eq: SMW vector formula}
		(A + UV^T)^{-1} = A^{-1} - \frac{A^{-1} \u \bv^T A^{-1}}{1 + \bv^T A^{-1} \u}.
	\end{equation}
%	in the special case that $k=1$ and $V^T A^{-1} U \neq -1$ .
\end{lemma}

Applying \eqref{eq: SMW vector formula} with
$A = r_q I$, $\u=\bv = \e$ shows that choosing
\begin{equation} \label{lambda formula}
	\lambda_{C_q} = \frac{1}{r_q} \left( W_{C_q, C_q} \e - \frac{1}{2} \left(\mu + \frac{ \e^T W_{C_q, C_q} \e}{r_q} \right) \e \right)
\end{equation}
ensures that $\tr(S_{C_q, C_q} X^*_{C_q, C_q}) =0$ for all $q = 1,\dots, k$.

% Parametrization of eta.
We next construct $\eta$.
Fix $q, s \in \{1, \dots, k+1\}$ such that $q \neq s$.
To ensure that $S_{C_q, C_s} \e = 0$ and $S_{C_s, C_q} \e =0$, we parametrize the entries of 
$\eta_{C_q, C_s}$ using the vectors $\y^{q,s}$ and $\z^{q,s}$.
In particular, we take
\begin{equation} \label{e: eta def}
	\eta_{C_q, C_s} = \left(\frac{\bar\delta_{q,k+1}}{2}\left( \alpha - \frac{\mu}{r_q} \right) + \frac{\bar \delta_{s,k+1}}{2}\left( \alpha - \frac{\mu}{r_s} \right) - \beta \right) \e\e^T + \y^{q,s} \e^T + \e(\z^{q,s})^T.
\end{equation}
Here $\bar \delta_{ij} := 1 - \delta_{ij}$, where $\delta_{ij}$  is the Kronecker delta function defined by $\delta_{ij} = 1$ if $i=j$ and $0$ otherwise.
That is, we take $\eta_{C_q, C_s}$ to be the expected value of $\lambda_{C_q} \e^T + \e \lambda_{C_s}^T - W_{C_q,C_s}$,
plus the parametrizing terms $\y^{q,s} \e^T$ and $\e (\z^{q,s})^T$.
The vectors $\y^{q,s}$ and $\z^{q,s}$ are chosen to be the solutions to the systems of linear equations imposed
by the requirement that $X^*S = SX^* = 0$.
As we will see, this system of linear equations is a perturbation of a linear system with known solution.
Using the solution of the perturbed system we obtain bounds on $\y^{q,s}$ and $\z^{q,s}$, which are  used
to establish that $\eta$ is nonnegative and $S$ is positive semidefinite.

Let 
\begin{equation}
	\label{e: tilde eta def}
	\tilde \eta_{C_q, C_s} := \lambda_{C_q} \e^T + \e \lambda_{C_s}^T - W_{C_q,C_s}.
\end{equation}
Note that the symmetry of $W$ implies that $\tilde \eta_{C_s, C_q} = \tilde \eta_{C_q, C_s}^T$.
Let $\b = \b^{q,s} \in \R^{C_q \cup C_s}$ be defined by
	$\b_{C_q} := \tilde \eta_{C_q, C_s} \e - \E[\tilde \eta_{C_q, C_s}] \e$
	and	
	$\b_{C_s} = \tilde \eta_{C_s, C_q} \e - \E[\tilde \eta_{C_s, C_q}] \e.$
We choose $\y = \y^{q,s}$  and $\z=\z^{q,s}$ to be solutions of the system
\begin{equation} \label{e: yz system}
	\mat{{cc} r_s I + \theta \e\e^T & (1-\theta) \e\e^T \\ (1-\theta) \e\e^T & r_q I + \theta \e\e^T}
	\vect{\y\\ \z} = \b
\end{equation}
for some scalar $\theta > 0$ to be defined later.
The requirement that the row sums of
$S_{C_q, C_s}$ are equal to zero is equivalent to $\y$ and $\z$ satisfying the system of linear equations
\begin{align} \notag 
	0 = -r_s \y_i -& \z^T \e  + r_s \left( \lambda_i - \frac{\bar \delta_{q,k+1}}{2 r_q}(\alpha r_q - \mu) \right)
		+ \left(\lambda_{C_s}^T \e - \frac{\bar\delta_{s,k+1}}{2}(\alpha r_s - \mu) \right) \\&
		- ([W_{C_q, C_s} \e]_i - r_s \beta)  \label{e: S row sum 1}
\end{align}
for all $i \in C_q$.
Similarly, the column sums of $S_{C_q, C_s}$ are equal to zero if and only if $\y$ and $\z$ satisfy
\begin{align} \notag
0 = -r_q \z_i -& \y^T \e  + r_q \left( \lambda_i - \frac{\bar \delta_{s,k+1}}{2 r_s}(\alpha r_s - \mu) \right)
		+ \left(\lambda_{C_q}^T \e - \frac{\bar \delta_{q,k+1}}{2}(\alpha r_q - \mu) \right) \\
		\label{e: S col sum 1}
		&- ([W_{C_s, C_q} \e]_i - r_q \beta) 
\end{align}
for all $i \in C_s$.
Note that the system of equations defined by \eqref{e: S row sum 1} and \eqref{e: S col sum 1} is equivalent to
\eqref{e: yz system} in the special case that $\theta = 0$.
However, when $\theta = 0$, the system of equations in \eqref{e: yz system} is singular, with nullspace
spanned by the vector $(\e; -\e)$. 
When $\theta$ is nonzero, each row of the system \eqref{e: yz system} has an additional term of the form
$\theta (\e^T \y - \e^T \z)$.
However, any solution $(\y; \z)$ of \eqref{e: yz system} for $\theta > 0$ is also a solution in the special case that $\theta = 0$.
Indeed, since $(\e; -\e)$ is in the nullspace of the matrix 
$$
	\mat{ {cc} r_s I & \e\e^T \\ \e\e^T & r_q I }
$$
and $\b_{C_q}^T \e = \b_{C_s} ^T \e$, 
taking the inner product of each side of \eqref{e: yz system} with $(\e; -\e)$ yields
$$
	\theta (r_q + r_s) (\e^T \y - \e^T \z) = \b_{C_q}^T \e - \b_{C_s}^T \e = 0.
$$
Therefore, the unique solution $(\y; \z)$ of \eqref{e: yz system}
also satisfies \eqref{e: S row sum 1} and\eqref{e: S col sum 1} for any $\theta > 0$ such that \eqref{e: yz system} is nonsingular.
%since the extra term $\theta (\e^T \y - \e^T \z)$ is zero.
%----------------------------------------------------------------------------------
% Special choice of theta = 1.
%----------------------------------------------------------------------------------
In particular, note that \eqref{e: yz system} is nonsingular for $\theta = 1$. For this choice
of $\theta$, $\y$ and $\z$ are the unique solutions of the systems
$(r_s I + \e\e^T) \y = \b_1$ and $(r_q I + \e\e^T) \z = \b_2$,
where $\b_1 := \b_{C_q}$ and $\b_2 := \b_{C_s}$.
Applying \eqref{eq: SMW vector formula} with $A= r_sI$, $\u=\bv=\e$ and
$A = r_qI$, $\u= \bv =\e$ yields 
\begin{align}
	\y &= \frac{1}{r_s} \left(\b_1 - \frac{(\b_1^T \e)}{r_q + r_s} \e\right) \label{e: y formula}  \;\;\mbox{ and } \;\;
	\z = \frac{1}{r_q} \left( \b_2 - \frac{(\b_2^T \e)}{r_q + r_s} \e \right) 
\end{align}
respectively.
% Choice of mu.
Finally, we choose $\mu = \epsilon \gamma \hat r$, where 
$
	\gamma = \gamma(\alpha, \beta, r) = \alpha (1 + \delta_{0, r_{k+1}} ) - 2 \beta,
$
and $\epsilon > 0$ is a scalar to be chosen later.

%----------------------------------------------------------------------------------
% Summary: choice of multipliers and Stilde.
%----------------------------------------------------------------------------------

In summary, we choose the multipliers $\mu \in \R$, $\lambda \in \R^{N}$, $\eta \in \R^{N\times N}$ as follows:
\begin{align}
	\mu = \epsilon \gamma\hat r \label{e: mu def} 
\end{align}
\begin{equation}
	\lambda_{C_q} = \branchdef{ \displaystyle{\frac{1}{r_q} \left( W_{C_q, C_q} \e - \frac{1}{2} \left(\mu + \frac{ \e^T W_{C_q, C_q} \e}{r_q} \right) \e \right)}, & \mbox{if } q \in \{1,\dots, k\} \\
				0, &\mbox{if } q = k+1  	}	\label{e: lambda def}  
\end{equation}
\begin{align} \label{e: eta def}
	\eta_{C_q, C_s} = \branchdef{ \E[\tilde \eta_{C_q, C_s}] + \y^{q,s} \e^T + \e (\z^{q,s})^T, & \mbox{if } q, s\in \{1,\dots, k+1\}, q\neq s \\
					0, &\mbox{otherwise} }
\end{align}
where $\epsilon>0$ is a scalar to be defined later, $\tilde \eta_{C_q, C_s}$ is defined as in \eqref{e: tilde eta def}, and $\y^{q,s}, \z^{q,s}$ are given by
\eqref{e: y formula}  for all $q, s \in \{1,\dots, k+1\}$ such that $q\neq s$.
We choose $S$ according to \eqref{E: w dual feas}.
Finally, we  define the $(k+1) \times (k+1)$ block matrix $\tilde S$ in $\Sigma^N$ by
\begin{equation} \label{e: tilde S def}
	\tilde S_{C_q, C_s} = \branchdef{ 	
						\alpha \e\e^T - W_{C_q, C_s},  & \mbox{if } q = s, \, q, s \in \{1,\dots, k\} \\
						\beta \e\e^T - W_{C_q, C_s},  & \mbox{if } q \neq s, \, q, s \in \{1,\dots, k\} \\
						\beta \e\e^T - W_{C_q, C_s} + (\lambda_{C_q} - \E[\lambda_{C_q}] )\e^T ,  & \mbox{if } s = {k+1}  \\
						\beta \e\e^T - W_{C_q, C_s} + \e (\lambda_{C_s} - \E[\lambda_{C_s}] )^T,   & \mbox{if } q = {k+1} .
						}						
\end{equation}

We conclude with the following theorem, which provides a sufficient  condition ensuring that the proposed 
solution $X^*$ is the unique optimal solution for \eqref{E: w trace relax} and  $K^*$ is the unique maximum density
$k$-disjoint-clique subgraph of $K_N$ corresponding to $W$.

%----------------------------------------------------------------------------------
% Optimality/Uniqueness
%----------------------------------------------------------------------------------
		
\begin{theorem} \label{T: S opt conds}
	Suppose that the vertex sets $C_1, \dots, C_k$ define a $k$-disjoint-clique subgraph $K^*$ of the complete graph $K_N = (V, E)$
	on $N$ vertices and let $C_{k+1} := V \setminus (\cup^k_{i=1} C_i).$
	Let $r_i := |C_i|$ for all $i=1,\dots, k+1$, and let $\hat r = \min_{i=1,\dots, k} r_i$.
	Let $W \in \Sigma^N$ be a random symmetric matrix sampled from the planted cluster model
	according to distributions  $\Omega_1, \Omega_2$ with means $\alpha, \beta$ satisfying \eqref{A: alpha beta ratio}.
	Let $X^*$ be the feasible solution for \eqref{E: w trace relax} corresponding to $C_1, \dots, C_k$ defined
	by \eqref{E: proposed sol}. Let $\mu \in \R,$ $\lambda \in \R^N,$  and $\eta \in \R^{N\times N}$
	be chosen according to \eqref{e: mu def}, \eqref{e: lambda def}, and \eqref{e: eta def},
	and let $S$ be chosen according to \eqref{E: w dual feas}.	
	Suppose that the entries of  $\lambda$ and  $\eta$ are nonnegative.
	Then there exists scalar $c > 0$  such that if
	%\begin{equation}
		$\| \tilde S\| \le c \gamma \hat r$
	%\end{equation}
	then $X^*$ is optimal for \eqref{E: w trace relax}, and $K^*$  is the maximum density $k$-disjoint-clique
	subgraph of $K_N$ corresponding to $W$.
	Moreover, if 
	\begin{equation} \label{a: block weights}
		r_s \e^T W_{C_q, C_q} \e > r_q \e^T W_{C_q, C_s} \e
	\end{equation}
	for all $q, s \in \{1,\dots,k\}$ such that $q\neq s$,
	then $X^*$ is the unique optimal solution of \eqref{E: w trace relax} and
	$K^*$ is the unique maximum density $k$-disjoint-clique
	subgraph of $K_N$.
\end{theorem}

\begin{proof}
	By construction,  $\mu$, $\lambda$, $\eta$, and $S$ satisfy \eqref{E: w dual feas}, \eqref{E: w CS rowsum},
	\eqref{E: w CS nonneg}, and \eqref{E: w CS sdp}.
	Moreover,  $\lambda$ and $\eta$ are nonnegative by assumption.
	Therefore,  it suffices to show that $S$ is positive semidefinite to prove that $X^*$ is optimal for \eqref{E: w trace relax}.
	To do so, we fix $\x \in \R^N$ and decompose $\x$ as $\x = \x_1 + \x_2$ where
	$$
		\x_1 (C_i) =
			\branchdef{ \phi_i \e, &\mbox{if } i \in \{1,\dots, k\} \\
					0, & \mbox{if } i = k + 1}
	$$
	for some $\phi \in \R^k$ chosen such that $\x_2(C_i)$ is orthogonal to $\e$ for all $i=1,\dots, k$,
	and $\x_2 (C_{k+1}) = \x(C_{k+1})$.
	Here, and in the rest of the note, the notation $\bv(A)$ denotes the vector in $\R^{|A|}$ with entries equal to those of $\bv$ indexed by $A$.
	Similarly, the notation $M(A,B)$ denotes the $|A|\times |B|$ matrix with entries equal those of  $M$
	indexed by $A$ and $B$ respectively.
	We have
	\begin{align*}
		\x^T S \x &= \x_2^T S \x_2 = \x_2^T (\tilde S + \mu I) \x_2 \ge \left( \mu - \|\tilde S\| \right) \|\x_2\| \ge \left( \epsilon \gamma \hat r - \|\tilde S\| \right) \|\x_2\| 			
	\end{align*}
	since $\x_2(C_i)$ is orthogonal to $\e$ for all $i=1,\dots, k$ and, hence,  $\x_2^T (S - \tilde S - \mu I) \x_2 = 0$.
	Therefore, if
	 $\|\tilde S\| \le \epsilon \gamma \hat r$, then $\x^T S \x \ge 0$ for all $\x \in \R^N$ with equality if and only if $\x_2 = 0$.
	In this case, $X^*$ is optimal for \eqref{E: w trace relax}.
	Moreover, $\bv_i$ is in the nullspace of $S$ for all $i=1,\dots, k$ by \eqref{E: w CS sdp} and  the fact that $X^* = \sum_{i=1}^k \bv_i \bv_i^T/r_i$.
	Since $\x^T S \x  = 0$ if and only if $\x_2 =0$, the nullspace of $S$ is exactly equal to the span of $\{\bv_1, \dots, \bv_k\}$
	and $S$ has rank equal to $N - k$.
	
	To see that $X^*$ is the unique optimal solution for \eqref{E: w trace relax} if Assumption~\eqref{a: block weights}
	holds, suppose, on the contrary, that $\tilde X$ is also optimal for \eqref{E: w trace relax}.
	By \eqref{E: w CS sdp}, we have $\tr(\tilde X S) = 0$, which holds if and only if $\tilde X S = 0$.
	Therefore, the row and column spaces of $\tilde X$ lie in the nullspace of $S$.
	Since $\tilde X \succeq 0$ and $\tilde X\ge 0$, we may write $\tilde X$ as
	\begin{equation} \label{e: tW svd}
		\tilde X = \sum_{i=1}^k \sum_{j=1}^k \sigma_{ij} \bv_i \bv_j^T		
	\end{equation}
	for some $\sigma \in \R^{k\times k}_+$.	
	The fact that $\tilde X$ satisfies $\tilde X \e \le \e$ implies that
	\begin{equation} \label{e: uniqueness row sum}
		\sigma_{qq} r_q + \sum_{\substack{s =1 \\ s\neq q}}^k \sigma_{qs} r_s \le 1
	\end{equation}
	for all $q=,1\dots, k$.
	Moreover, since $\tr(W\tilde X) = \tr(W X^*)$, there exists some $q \in \{1,\dots, k\}$
	such that
	\begin{equation} \label{e: uniqueness values 1}
		\sigma_{qq} \bv_q^T W \bv_q + \sum_{\substack{s =1 \\ s\neq q}}^k \sigma_{qs} \bv_q^T W \bv_s \ge \frac{\bv_q^T W \bv_q}{r_q}.
	\end{equation}
	Combining \eqref{e: uniqueness row sum} and \eqref{e: uniqueness values 1} shows that
	\begin{align*}
		0 &\le \bv_q^T W \bv_q \left( \frac{1}{r_q} - \sum_{\substack{s =1 \\ s\neq q}}^k \frac{ \sigma_{qs} r_s}{r_q}\right)
			+ \sum_{\substack{s =1 \\ s\neq q}}^k \sigma_{qs} \bv_q^T W \bv_s -\frac{\bv_q^T W \bv_q}{r_q} \\
		&= \sum_{\substack{s =1 \\ s\neq q}}^k \frac{\sigma_{qs}}{r_q} (r_q \bv_q^T W \bv_s - r_s \bv_q^T W \bv_q),
	\end{align*}	
	contradicting Assumption~\eqref{a: block weights}.
	Therefore, $X^*$ is the unique optimal solution of \eqref{E: w trace relax}
	as required. \qed
\end{proof}

%\subsection{A sufficient condition for exact recovery for the WKDB problem}

\subsection{Nonnegativity of $\bold{\lambda}$ and $\bold{\eta}$ in the planted case}
\label{sec: c nonneg}
% Preamble.
We now establish that the entries of $\lambda$ and $\eta$ are nonnegative with probability tending exponentially
to $1$ as $\hat r$ approaches $\infty$ for sufficiently small choice of $\epsilon$ in \eqref{e: mu def}.

% Hoeffding.
We begin by deriving lower bounds on the entries of $\eta$.
To do so, we will repeatedly apply the following theorem of Hoeffding (see \cite[Theorem 1]{Hoeffding}),
which provides a bound on the tail distribution of a sum of bounded, independent random variables.

\begin{theorem}[Hoeffding's Inequality]
	\label{thm: Hoeffding}
	Let $X_1, \dots, X_m$ be \iid variables sampled from a distribution satisfying $0 \le X_i \le 1$ for all $i = 1,\dots, m$.
	Let $S = X_1 + \cdots + X_m$. Then
	\begin{equation}
		\label{Hoeffding-1}
		Pr( |S - \E[S]| > t) \le 2 \exp \left(\frac{-2t^2}{m}\right)
	\end{equation}
	for all $t > 0$.
%	In particular, for $t = a\sqrt{m}$ for some scalar $a > 0$, we have
%	\begin{equation} \label{Hoeffding-sqrt}
%		Pr ( |S - E[S]| > a \sqrt{m}) \le 2 \exp(-2 a^2).
%	\end{equation}
\end{theorem}

% y, z inf norm bounds.
To show that $\eta_{ij} \ge 0$ for all $i,j\in V$ with high probability,
we will use the following lemma, which provides an upper bound on $\|\y^{q,s}\|_\infty$ and $\|\z^{q,s}\|_\infty$
 for all $q,s \in \{1,\dots,k+1\}$ such that $q\neq s$,
 holding with probability tending to $1$ as $\hat r$ tends to $\infty$.

\begin{lemma} \label{l: y, z inf bound}
	There exists scalar $\tilde{c} > 0$ such that
	%\begin{equation} \label{e: y,z inf bound}
	$
		\|\y^{q,s}\|_\infty + \|\z^{q,s}\|_\infty \le \tilde c \hat r^{-1/4}
	$	
	%\end{equation}
	for all $q,s \in \{1,\dots, k+1\}$ such that $q\neq s$
	with \pteto{\hat r}.
\end{lemma}

\begin{proof}
Fix $q, s\in \{1, \dots, k\}$ such that $q \neq s$.
Without loss of generality, we assume that $r_q \le r_s$.
The proof for the case when either $q$ or $s$ is equal to $k+1$ is analogous.
We first obtain an upper bound on $\|\y\|_\infty = \|\y^{q,s}\|_\infty$.
By the triangle inequality, we have
\begin{equation} \label{e: y inf bound 1}
	\|\y\|_\infty \le \frac{1}{r_s}\left\|\b_1 + \frac{|\b_1^T \e|}{r_q + r_s} \e \right\|_\infty
		\le \frac{1}{r_s} \left( \|\b_1\|_\infty + \frac{|\b_1^T \e|}{r_q + r_s} \right).
\end{equation}
Hence, to obtain an upper bound on  $\|\y\|_\infty$, it suffices to obtain  bounds on
$\|\b_1\|_\infty$ and $|\b_1^T\e|$.
% \|b1\|_inf bound
We begin with $\|\b_1\|_\infty$.
Recall that we have 
\begin{equation} \label{e: bi formula}
	\b_i = r_s \left( \lambda_i - \frac{1}{2r_q} (\alpha r_q - \mu) \right) + \left( \lambda_{C_s}^T \e - \frac{1}{2}(\alpha r_s - \mu) \right) - \left( \sum_{j \in C_s} W_{ij} - \beta r_s \right).
\end{equation}
for each $i \in C_q$.
Note that
\begin{align*}
	\lambda_{C_s}^T \e = \frac{1}{r_s} \left( \e^T W_{C_s, C_s} \e - \frac{1}{2} r_s \mu - \frac{1}{2} \e^T W_{C_s,C_s} \e \right)
		= \frac{1}{2r_s}( \e^T W_{C_s,C_s} \e - r_s \mu).
\end{align*}
Applying \eqref{Hoeffding-1} with $S = \tr(W_{C_s, C_s})$, $t = r_s^{3/2}$ and $S = \sum_{i\in C_s} \sum_{j \in C_s,  j > i}  W_{ij}$, $t = r_s^{3/2}/2$  shows that
\begin{align} 
	&\left|\lambda_{C_s}^T \e  -  \frac{1}{2}(\alpha r_s - \mu) \right|
		= \frac{1}{2r_s} |\e^T W_{C_s, C_s} \e - \alpha r_s^2| \notag\\
		&\le \frac{1}{2 r_s} \rbra{ |\tr(W_{C_s, C_s}) - \alpha r_s | + 2 \left| \sum_{i\in C_s} \sum_{\substack{j \in C_s \\ j > i}  } W_{ij} -  \frac{ \alpha r_s(r_s-1)}{2} \right| } %\notag\\ &
		\le \sqrt{r_s} \label{e: bi term-2}
\end{align}
with probability at least
%\begin{equation} \label{e: bi term-2 prob}
$1 - 2\exp(-2 r_s^2) - 2\exp(- r_s^2/(r_s-1)) \ge 1 - \tilde p_1,$
%\end{equation}
 where 	$\tilde p_1 := 2 \exp(-2 \hat r^2) +  2 \exp(-\hat r)$.
Next, applying \eqref{Hoeffding-1} with $S = \sum_{\ell\in C_s} W_{i\ell}$ and $t = r_s^{3/4}$ shows
that
\begin{equation} \label{e: bi term 3}
	\left| \sum_{\ell \in C_s} W_{i\ell} - \beta r_s \right| \le r_s^{3/4}
\end{equation}
with probability at least $1 - \tilde p_2$ where
$\tilde p_2 := 2 \exp(-2 \hat r^{1/2}).$
Finally, applying \eqref{Hoeffding-1} with $S = \sum_{\ell\in C_q} W_{i\ell},$ $t = r_q^{3/4}$
and \eqref{e: bi term-2} shows that
\begin{align}
	\left|\lambda_i  - \frac{1}{2r_q} (\alpha r_q - \mu)\right|
		 &\le \frac{1}{r_q} \left| \sum_{\ell\in C_q} W_{i\ell} - r_q \alpha \right|
		+ \frac{1}{2 r_q^2} \left| \e^T W_{C_q, C_q} \e - \alpha r_q^2 \right| \notag\\
		&\le r_q^{-1/4} + r_q^{-1/2}  \le 2r_q^{-1/4} \label{e: bi term 1}
\end{align}
with probability at least $1 - \tilde p_1 - \tilde p_2$.
%Finally, by applying \eqref{Hoeffding-1} with $S = \sum_{\ell\in C_q} W_{i\ell},$ $t = r_q^{3/4}$
%and \eqref{eqn: Bernstein} with $\delta=1$, $I_1=I_2 = C_q$ and $X_{\ell_1,\ell_2} = W_{\ell_1,\ell_2}$ for all $\ell_1,\ell_2\in C_q$.
%\begin{align}
%	|\lambda_i  - \frac{1}{2r_q} (\alpha r_q - \mu)|
%		 &\le \frac{1}{r_q} \left| \sum_{\ell\in C_q} W_{i\ell} - r_q \alpha \right|
%		+ \frac{1}{2r_q^2} \sum_{\ell_1\in C_q} \left| \sum_{\ell_2\in C_q} W_{\ell_1, \ell_2} - r_q \alpha \right| \notag\\
%		&\le r_q^{-1/4} + \frac{B}{2} r_q^{-1/2}  \le \left(1 + \frac{B}{2}\right) r_q^{-1/4} \label{e: bi term 1}
%\end{align}
%with probability at least $1 - \tilde p_1 - \tilde p_2$.
Combining \eqref{e: bi term-2}, \eqref{e: bi term 3} and \eqref{e: bi term 1} and applying the union bound shows that
\begin{equation} \label{e: b1 inf norm}
	\|\b_1\|_\infty \le  4 r_q^{-1/4} r_s
\end{equation}
with probability at least $1 - \tilde p_1 - 2 r_q \tilde p_2.$
By a similar argument, 
$%\begin{equation} \label{b2 inf norm}
	\|\b_2\|_\infty \le  4 r_q^{3/4}
$%\end{equation}
with probability at least $1 - \tilde p_1 - 2 r_s \tilde p_2.$
%Combining \eqref{e: bi term-2}, \eqref{e: bi term 3} and \eqref{e: bi term 1} and applying the union bound shows that
%there exists scalar $c_1>0$ such that
%\begin{equation} \label{e: b1 inf norm}
%	\|\b_1\|_\infty \le  c_1 r_q^{-1/4} r_s
%\end{equation}
%with probability at least $1 - (r_q + 1)(\tilde p_1 + \tilde p_2).$
%By a similar argument, there exists scalar $c_2 > 0$ such that
%\begin{equation} \label{b2 inf norm}
%	\|\b_2\|_\infty \le  c_2 r_q^{3/4}
%\end{equation}
%with probability at least $1 - (r_s + 1)(\tilde p_1 + \tilde p_2).$

% |b1^T e| bound.
We next obtain an upper bound on $|\b_1^T \e|$ and $|\b_2^T \e|$.
We have
\begin{equation} \label{b1 e expansion}
	\b_1^T \e = r_s \left( \lambda_{C_q}^T \e - \frac{1}{2} (\alpha r_q - \mu) \right) 
		+ r_q \left( \lambda_{C_s}^T \e - \frac{1}{2} (\alpha r_s - \mu) \right) + (\beta r_s r_q - \e^T W_{C_q, C_s} \e).
\end{equation}
By \eqref{e: bi term-2} and the union bound, we have
\begin{align}
	\left|\lambda_{C_s}^T \e - \frac{1}{2}(\alpha r_s - \mu) \right| &\le  \sqrt{r_s} \label{v2 1}  \\
	\left|\lambda_{C_q}^T \e - \frac{1}{2}(\alpha r_q - \mu) \right| &\le  \sqrt{r_q} \label{v2 2}
\end{align}
with probability at least $1 - 2 \tilde p_1$.
Moreover, applying \eqref{Hoeffding-1} with $S = \e^T W_{C_q, C_s} \e$ and $t = r_q \sqrt{r_s}$
shows that
\begin{equation} \label{v3 bound}
	|\e^T W_{C_q, C_s} \e - \beta r_s r_q| \le  r_q \sqrt{r_s}
\end{equation}
with probability at least $1 -\tilde p_3$, where $\tilde p_3 := 2 \exp(- 2 \hat r).$
Substituting \eqref{v2 1} and \eqref{v3 bound} into \eqref{b1 e expansion}, we have 
\begin{equation} \label{|b1 e|}
	|\b_1^T \e| \le  3 r_s \sqrt{r_q}
\end{equation}
for some scalar $c_3>0$ with probability 
at least $1 - 2 \tilde p_1 - \tilde p_3$ by the union bound.
%Substituting \eqref{v2 1}, \eqref{v2 2}, and \eqref{v3 bound} into \eqref{b1 e expansion}, we have 
%\begin{equation} \label{|b1 e|}
%	|\b_1^T \e| \le  c_3 r_s \sqrt{r_q}
%\end{equation}
%for some scalar $c_3>0$ with probability 
%at least $1 - 3 \tilde p_1$ by the union bound.
Similarly, 
\begin{equation} \label{|b2 e|}
	|\b_2^T \e| \le 3 r_s \sqrt{r_q}
\end{equation}
with probability at least $1 - 2 \tilde p_1 - \tilde p_3$.
%Similarly, there exists scalar $c_4 > 0$ such that
%\begin{equation} \label{|b2 e|}
%	|\b_2^T \e| \le c_4 r_q \sqrt{r_s}
%\end{equation}
%with probability at least $1 - 3 \tilde p_1$.
Substituting \eqref{e: b1 inf norm} and \eqref{|b1 e|} in \eqref{e: y inf bound 1}
yields
\begin{equation} \label{e: y inf bound}
	\| \y\|_\infty \le  \tilde c_1 r_q^{-1/4 },
\end{equation}
for some scalar $\tilde c_1 > 0$,
with probability at least 
$
	1 - (3 \tilde p_1 + 2 r_q \tilde p_2 + \tilde p_3).
$
%\begin{equation} \label{e: y,z inf prob}
%	1 - (r_q + 1)(\tilde p_1 + \tilde p_2) - 3 \tilde p_1 \ge 1 - (N+4) \tilde p_1 - (N+1) \tilde p_2.
%\end{equation}
Similarly, there exists scalar $\tilde c_2 > 0$ such that
\begin{equation} \label{e: z inf bound}
	\|\z\|_\infty \le \tilde c_2 r_q^{-1/4}
\end{equation}
with probability at least 
$
	1 - (3 \tilde p_1 + 2 r_s \tilde p_2 + \tilde p_3).
$
Combining \eqref{e: y inf bound} and \eqref{e: z inf bound} and applying the union bound over all $q,s$  completes the proof.
\qed
\end{proof}
\vspace{0.15in}

% Nonnegativity of eta cor:
As an immediate consequence of Lemma~\ref{l: y, z inf bound}, %we have the following corollary that states that 
$\eta$ is nonnegative with probability tending exponentially to $1$ for sufficiently large
values of $\hat r$.

\begin{cor}\label{C: eta nonneg}
	Suppose that $\alpha$ and $\beta$ satisfy \eqref{A: alpha beta ratio}.
	Then the entries of the matrix $\eta$ are nonnegative with
	probability tending exponentially to $1$ as $\hat r$ approaches $\infty$.
\end{cor}

\begin{proof}
Fix  $i \in C_q$, $j \in C_s$ for some $q, s\in\{1, \dots, k+1\}$ such that $q\neq s$.
Recall that
$$
	\eta_{C_q, C_s} = \left(\frac{\bar\delta_{q,k+1}}{2}\left( \alpha - \frac{\mu}{r_q} \right) + \frac{\bar \delta_{s,k+1}}{2}\left( \alpha - \frac{\mu}{r_s} \right) - \beta \right) \e\e^T + \y^{q,s} \e^T + \e(\z^{q,s})^T.
$$
Therefore,
since $\gamma > 0$ by \eqref{A: alpha beta ratio}, Lemma~\ref{l: y, z inf bound} implies that
\begin{align*}
	\eta_{ij} &\ge \frac{(1-\delta_{q,k+1})}{2}\rbra{\alpha - \frac{\mu}{r_q}} + \frac{(1-\delta_{s,k+1})}{2}\rbra{\alpha - \frac{\mu}{r_s}}
				- \beta - \|\y^{q,s}\|_\infty - \|\z^{q,s}\|_\infty \\
%	& \ge \frac{\alpha}{2} (2-\delta_{q,k+1} - \delta_{s,k+1}) -  \beta - \epsilon \gamma - \tilde c \hat r^{-1/4}\\
	&\ge   \rbra{ \frac{1}{2} - \epsilon} \gamma - \tilde c \hat r^{-1/4}	\ge 0,
\end{align*}
for all sufficiently small $\epsilon > 0$ and sufficiently large $\hat r$
with \pteto{\hat r},
since at most one of $q$ and $s$ is equal to $k+1$.
\qed
\end{proof}
\vspace{0.15in}

% Nonnegativity of Lambda
The following lemma provides a similar lower bound on the entries of $\lambda$.

\begin{lemma} \label{T: lambda bound}
	There exist scalars $\bar c_1, \bar c_2 > 0$ such that
	%\begin{equation} \label{e: lambda bound}
	$	\lambda_{i} \ge \hat r (\bar c_1 - \bar c_2 \hat r^{-1/4} ) $
	%\end{equation}
	for all $i \in V \setminus C_{k+1}$
	with \pteto{\hat r}.
\end{lemma}

\begin{proof}
Fix $q \in \{1,\dots,k\}$ and $i \in C_q$. Recall that 
$$
	\lambda_i = \sum_{j \in C_q} W_{ij} - \frac{1}{2r_q} \e^T W_{C_q, C_q} \e - \frac{\mu}{2}.
$$
Applying \eqref{Hoeffding-1} with $S = \sum_{j\in C_q} W_{ij}$ and $t=r_q^{3/4}$ yields
\begin{equation} \label{e: W row sum bound}
	\sum_{j \in C_q} W_{ij} \ge \alpha r_q - r_q^{3/4}
\end{equation}
with probability at least $1 - \tilde p_2$.
Moreover, \eqref{v2 2} implies that
\begin{equation} \label{e: W block sum bound}
	\frac{1}{2r_q} \e^T W_{C_q, C_q} \e \le \frac{1}{2}(\alpha r_q +  \sqrt{r_q})
\end{equation}
%\begin{equation} \label{e: W block sum bound}
%	\frac{1}{2r_q} \e^T W_{C_q, C_q} \e \le \frac{1}{2}(\alpha r_q +  B\sqrt{r_q})
%\end{equation}
with probability at least $1 - \tilde p_1$.
Combining \eqref{e: W row sum bound} and \eqref{e: W block sum bound} and applying the union bound shows
that there exist scalars $\bar c_1, \bar c_2 > 0$  such that
\begin{align*}
	\lambda_i &\ge  \alpha r_q - r_q^{3/4} - \frac{1}{2}( \alpha r_q +  \sqrt{r_q}) - \frac{\mu}{2}  	%\notag\\&
		\ge r_q (\bar c_1 - \bar c_2 r_q^{-1/4} )
\end{align*}
%\begin{align*}
%	\lambda_i &\ge  \alpha r_q - r_q^{3/4} - \frac{1}{2}( \alpha r_q +  B\sqrt{r_q}) - \frac{\mu}{2}  	%\notag\\&
%		\ge r_q (\bar c_1 - \bar c_2 r_q^{-1/4} )
%\end{align*}
with probability at least $1 - \tilde p_1 - \tilde p_2$ for sufficiently small choice of $\epsilon > 0$ in \eqref{e: mu def}.
 Applying the union bound over all
$i \in  V \setminus C_{k+1}$ completes the proof.
\qed
\end{proof}
\vspace{0.15in}

Note that Lemma~\ref{T: lambda bound} implies that $\lambda \ge 0$ with probability tending exponentially to $1$ as $\hat r$ tends to $\infty$.
Therefore, $\mu$, $\lambda$, $\eta$ constructed according to \eqref{e: mu def}, \eqref{e: lambda def}, and \eqref{e: eta def} are dual feasible for \eqref{E: w trace relax}
with \pteto{\hat r} if the left-hand side of \eqref{E: w dual feas}  is positive semidefinite.
The following lemma states the uniqueness condition given by \eqref{a: block weights} is also satisfied with high probability for sufficiently
large $\hat r$.

\begin{lemma} \label{T: WKDC uniqueness}
	If $\hat r > 9/(\alpha - \beta)^2$ then
%	There exists scalar $c > 0$ such that if $\hat r > 4c/(\alpha - \beta)^2$ then
	$
		r_s \e^T W_{C_q, C_q} \e > r_q \e^T W_{C_q, C_s} \e
	$
	for all $q,s \in \{1,\dots, k\}$ such that $q \neq s$ with \pteto{\hat r}.
\end{lemma}

\begin{proof}
	Fix $q \neq s$ such that $r_q \le r_s$.
%	Recall that
%	\begin{equation} \label{e: W uni proof 1}
%		\e^T W_{C_q, C_q} \e \ge \alpha r_q^2 - 2 r_q^{3/2}
%	\end{equation}
%	with probability at least $1- \tilde p_1$ by \eqref{e: bi term-2}.
%%	Applying \eqref{eqn: Bernstein} with $\delta = 1$, $I_1 = I_2 = C_q$, and $X_{\ell_1,\ell_2} = W_{\ell_1,\ell_2}$ for all $\ell_1,\ell_2\in C_q$
%%	shows that
%%	\begin{equation} \label{e: W uni proof 1}
%%		\e^T W_{C_q, C_q} \e \ge \alpha r_q^2 - B r_q^{3/2}
%%	\end{equation}
%%	with probability at least $1- \tilde p_1$.
%	Similarly,
%	\begin{equation} \label{e: W uni proof 2}
%		\e^T W_{C_q, C_s} \e \le \beta r_q r_s +  r_q r_s^{1/2}
%	\end{equation}
%	with probability at least $1 - \tilde p_3$ by \eqref{v3 bound}.
%%	Similarly, applying \eqref{eqn: Bernstein} with $\delta = 1$, $I_1 = C_q$, $I_2 = C_s$, and $X_{\ell_1,\ell_2} = W_{\ell_1,\ell_2}$ for all $\ell_1 \in C_q,\ell_2\in C_s$
%%	yields
%%	\begin{equation} \label{e: W uni proof 2}
%%		\e^T W_{C_q, C_s} \e \le \beta r_q r_s + B r_q r_s^{1/2}
%%	\end{equation}
%%	with probability at least $1 - \tilde p_1$.
	Combining \eqref{e: bi term-2} and \eqref{v3 bound}
	shows that
	\begin{align*}
		r_s \e^T W_{C_q, C_q} \e - r_q \e^T W_{C_q, C_s} \e  &\ge r_s r_q^2( \alpha - \beta - 2 r_q^{-1/2} - r_s^{-1/2}) %\\
			\ge r_s r_q^2(\alpha - \beta - 3 \hat r^{-1/2})
	\end{align*}
	with probability at least $1- \tilde p_1 - \tilde p_3$. Noting that this lower bound is positive if $\hat r > 9/(\alpha-\beta)^2$ and
%	Combining \eqref{e: W uni proof 1} and \eqref{e: W uni proof 2}
%	yields
%	\begin{align*}
%		r_s \e^T W_{C_q, C_q} \e - r_q \e^T W_{C_q, C_s} \e  &\ge r_s r_q^2( \alpha - \beta - B(r_q^{-1/2} + r_s^{-1/2}) \\
%			&\ge r_s r_q^2(\alpha - \beta - 2B \hat r^{-1/2}) > 0
%	\end{align*}
%	if $\hat r > 4B^2/(\alpha-\beta)^2$, with probability at least $1-2 \tilde p_1$.
	applying the union bound over all choices of $q$ and $s$ completes the proof.
	\qed
%	\bigskip
\end{proof}

We have shown that $\mu, \lambda, \eta$ constructed according to \eqref{e: mu def}, \eqref{e: lambda def}, and \eqref{e: eta def} are dual feasible for \eqref{E: w trace relax}
and the uniqueness condition \eqref{a: block weights} is satisfied
with \pteto{\hat r}.			
In the next subsection, we derive an upper bound on
the norm of $\tilde S$ and use this bound to obtain conditions ensuring dual feasibility of $S$ and, hence,
optimality of $X^*$ for \eqref{E: w trace relax}.

\subsection{An upper bound on $\|\tilde S\|$}
% Statement of assumptions.
In this section, we derive an upper bound on $\|\tilde S\|$, which will
be used to verify that the conditions on the partition $C_1, \dots, C_{k+1}$ 
imposed by \eqref{A: wkdc guarantee bound}
ensure that the \kdcs of $K_N$ composed of the cliques $C_1, \dots, C_k$ is the
unique maximum density \kdc of $K_N$ with respect to $W$ and can be
recovered by solving \eqref{E: w trace relax} with \pteto{\hat r}.
In particular, we will prove the following lemma.

% S bound thm.
\begin{lemma} \label{thm: S bound}	
	There exist scalars $\rho_1, \rho_2 > 0$ such that
	\begin{equation} \label{eqn: S bound}
		\| \tilde S\| \le \rho_1 \sqrt{N}
			+ \rho_2 \sqrt{ k r_{k+1}} + \beta r_{k+1}
	\end{equation}
	with probability tending exponentially to $1$ as $\hat r$ approaches $\infty$.	
\end{lemma}

This lemma, along with 
Theorem~\ref{T: S opt conds}, Lemma~\ref{T: lambda bound}, and 
Corollary~\ref{C: eta nonneg}, establishes Theorem~\ref{T: weighted kdc results}.
Indeed, if the right-hand side of \eqref{eqn: S bound} is bounded above by $O (\gamma \hat r)$ 
then Theorem~\ref{T: S opt conds}, Lemma~\ref{T: lambda bound}, and 
Corollary~\ref{C: eta nonneg} imply that the \kdcs given by $C_1, \dots, C_k$ is the densest \kdcs
corresponding to $W$ and can be recovered by solving \eqref{E: w trace relax}.

% tilde S decomp.
The remainder of this section consists of a proof of Lemma~\ref{thm: S bound}.
We decompose $\tilde S$ as
$
     \tilde S = \tilde S_1 + \tilde S_2 + \tilde S_3
$
where $\tilde S_i \in \Sigma^N$, $i=1,\dots, 3$, are $(k+1)$ by $(k+1)$ block matrices such that
\begin{align*}
     \tilde S_1(C_q, C_s) &= \E[W] - W \\
     \tilde S_2(C_q, C_s) &= \left\{ \begin{array}{ll}
     				(\lambda_{C_q} - \E[\lambda_{C_q}]) \e^T, &\mbox{if } s = k+1 \\
     				\e (\lambda_{C_s} - \E[\lambda_{C_s}])^T, &\mbox{if } q =  k+1 \\
                                             0 & \mbox{otherwise} \\
                                   \end{array}
                                   \right. \\
    \tilde S_3(C_q, C_s) &= \left\{ \begin{array}{ll}
                                             -\beta \e\e^T,&\mbox{if } q = s = k+1 \\
                                             0, & \mbox{otherwise.} \\
                                   \end{array}
                                   \right.
\end{align*}
% F& K bound.
To bound the norm of each matrix in this decomposition, we will make repeated use of the following bound
on the norm of  a random symmetric matrix (see \cite{Furedi-Komlos:1981}, \cite[Theorem~1]{Ames-Vavasis}).

\begin{theorem}
    \label{Furedi-Komlos}
        Let $A \in \Sigma^n$ be a random symmetric matrix with i.i.d.~entries sampled from a distribution 
        with mean $\mu$ and variance $\sigma^2$ such that $A_{ij} \in [0,1]$ for all $i,j \in \{1,\dots, n\}$.
        Then
	$
            	\| A - \mu \e\e^T \| \le 3 \sigma\sqrt{n}
	$
	with probability at least $1 - \exp(-c  n^{1/6})$ where $c$ depends only on $\sigma$.
\end{theorem}

%This theorem is not stated exactly in this
%manner in \cite{Furedi-Komlos:1981}, 
%but can be deduced by taking $k = \sigma^{1/3} n^{1/6}$ and $v = \sigma \sqrt{n}$
%in the inequality
%$$
%	P(\max |\lambda| > 2 \sigma\sqrt{n} + v) < \sqrt{n} \exp(-kv/(2 \sigma \sqrt{n} + v))
%$$
%on pp. 237 of  \cite{Furedi-Komlos:1981}.

We are now ready to compute the desired bound on $\|\tilde S\|$.
% Bound on S1.
By Theorem~\ref{Furedi-Komlos}, there exist $\rho_1 > 0$ such that
$
	\|\tilde S_1 \| \le \rho_1 \sqrt{N}
$
with \pteto{N}.
Morever, we have
$
	\|\tilde S_3\| = \beta \|\e\e^T\| = \beta r_{k+1}.
$
It remains to obtain an upper bound on $\|\tilde S_2\|$.

%% Bound on S2.
Note that $\| \tilde S_2 \| \le 2 \|\lambda - \E[\lambda] \| \sqrt{r_{k+1}} $ by the triangle inequality.
Recall that
$$
	\lambda_{C_q} - \E[\lambda_{C_q}] = \frac{1}{r_q} \rbra{ \rbra{ W_{C_q, C_q} \e - \alpha r_q \e } - \frac{1}{r_q} \rbra{ \e^T W_{C_q, C_q} \e - \alpha r_q^2 } \e }
$$
for all $q = 1,\dots, k$.
Applying Theorem~\ref{Furedi-Komlos}, there exists $\varphi > 0$ such that
$$
	\| W_{C_q, C_q} \e - \alpha r_q\e \| \le \| W_{C_q, C_q} - \alpha \e\e^T \| \|\e\| \le \varphi r_q
$$
with \pteto{\hat r}.
On the other hand,
\eqref{e: bi term-2} implies that $|\e^T W_{C_q, C_q} \e - \alpha r_q^2 | \le 2 r_q^{3/2}$ with probability at least $1 - \tilde p_1$.
It follows that there exists scalar $\rho_2 > 0$ such that
$
	\|\lambda_{C_q} - \E[ \lambda_{C_q} ] \|^2 \le \rho_2^2/4
$
for all $q = 1,\dots, k$
with \pteto{\hat r}. Therefore,
$
	\|\lambda - \E[\lambda]\|^2 = \sum_{q=1}^k \|\lambda_{C_q} - \E[ \lambda_{C_q} ] \|^2 \le k \rho_2^2/4
$
with high probability, as required. This completes the proof of Lemma~\ref{thm: S bound}.

%%======================================================================
\section{Proof of Theorem~\ref{thm: biclustering recovery guarantee}}
\label{sec: kdb proof}
%%======================================================================
\subsection{Optimality conditions and choice of multipliers}
\label{sec: b opt conds}
% Relation to WKDC proof.
%The proof of  follows a  trajectory to similar that of Theorem~\ref{T: weighted kdc results}.
We provide of a sketch of the proof of Theorem~\ref{thm: biclustering recovery guarantee} here; many of the technical details are identical to those in the proof of Theorem~\ref{T: weighted kdc results} and are omitted.
%In this section, we provide conditions for optimality of the proposed optimal solution $Z^*$ of the semidefinite relaxation of the
%densest \kdb problem given by \eqref{eq: WKDB relaxation}.
% KKT conditions
As before, we will establish that a proposed solution satisfies a set of sufficient conditions for optimality for \eqref{eq: WKDB relaxation}, given by the following theorem,
with high probability if the input graph satisfies the assumptions of Theorem~\ref{thm: biclustering recovery guarantee}.
%The following theorem provides the req optimality conditions to \eqref{eq: WKDB relaxation}.

\begin{theorem}
	\label{thm:  WKDB KKT conditions}
	Let 	$Z$
	be feasible for \eqref{eq: WKDB relaxation} and suppose that 
	there exist some $\mu_1, \mu_2 \in \R$, $\lambda \in \R^M_+$, $\phi \in \R^N_+$, $\eta \in \R^{(M+N)\times (M+N)}_{+}$
	and $S \in \Sigma^{M+N}_+$ such that
	\begin{align}
		\mat{{cc} \mu_1 I +\lambda \e^T + \e \lambda^T & -W \\ - W^T & \mu_2 I + \phi \e^T  + \e \phi^T}		
		 - \eta  &= S \label{eq: WKDB dual feas} \\
		\lambda^T (Z_{U,U} \e - \e) &= 0  \label{eq: WKDB CS U rowsum} \\		
		\phi^T (Z_{V,V} - \e) &=0 \label{eq: WKDB CS V rowsum} \\
		\tr( Z \eta) &= 0 \label{eq: WKDB CS nonneg} \\
		\tr(Z S)  &= 0. \label{eq: WKDB CS sdp}
	\end{align}
	Then $Z$ is optimal for \eqref{eq: WKDB relaxation}.
\end{theorem}

Let $(U_1,V_1)$, \dots, $(U_k,V_k)$ denote the vertex sets of  the \kdb subgraph $G^*$ of the 
bipartite complete graph $K_{M,N} = ((U,V),E)$ with vertex sets $U$ and $V$ of size $M$ and $N$ respectively.
Let $U_{k+1} := U \setminus (\cup_{i=1}^k U_i)$ and $V_{k+1} := V \setminus (\cup_{i=1}^k V_i)$.
Let $W \in \R^{M\times N}$ be a random nonnegative matrix sampled from the planted bicluster model
according to distributions $\Omega_1, \Omega_2$ with means $\alpha, \beta$.
Let $m_i := |U_i|$, $n_i := |V_i|$ for all $i=1,\dots, k+1$, and let $\hat m = \min_{i=1,\dots,k} m_i$, $\hat n = \min_{i=1,\dots,k} n_i$.
Let $C_i := U_i \cup V_i$ and let $r_i := |C_i| = m_i + n_i$ for all $i=1,\dots, k+1$.
We assume that $m_i$ is equal to a scalar multiple $\tau_i^2$ of $n_i$ for all $i \in \{1,\dots, k+1\}$. That is,
$m_i = \tau_i^2 n_i$ for some $\tau_i >0$ for all $i=1,\dots, k+1$.

%We next by provide necessary background regarding the norms 
%of random matrices.

As before, we establish optimality of $Z^*$ by constructing dual multipliers satisfying the assumptions of Theorem~\ref{thm:  WKDB KKT conditions}.
% Choice of lambda, phi
The matrix $S$ and, hence, $\lambda$, $\phi$, and $\eta$ will be constructed in blocks indexed by the vertex sets
$U_1,\dots, U_{k+1}$ and $V_1,\dots, V_{k+1}$.
Note that the diagonal blocks of $Z_{U,U}^*$ indexed by $U_1,\dots, U_k$ consist of multiples of the all-ones matrix and the remaining
blocks are equal to $0$. 
Therefore, $\lambda_{U_{k+1}} =0$ by \eqref{eq: WKDB CS U rowsum}.
Similarly, the block structure of $Z^*$ implies that $\phi_{V_{k+1}} = 0$ by \eqref{eq: WKDB CS V rowsum} and
$\eta_{C_q, C_q} = 0$ for all $q=1,\dots, k$ by \eqref{eq: WKDB CS nonneg}.

% Lambda formula.
Since both $S$ and $Z^*$ are assumed to be positive semidefinite matrices, the complementary slackness condition, $\tr(Z^*S) = 0$, is equivalent to requiring
the columns of $Z^*$ to reside in the nullspace of $S$.
For each $q \in \{1,\dots, k\},$ we choose $\lambda_{U_q}$ so that $S_{U_q, C_q}$ is orthogonal to $Z^*_{U_q, C_q}$.
In particular, it suffices to choose $\lambda$ such that
\begin{equation} \label{eq: lambda eq}
	0 	= S_{U_q, U_q} \e +  \tau_q S_{U_q, V_q} \e
		= \mu_1 \e + m_q \lambda_{U_q} + (\lambda_{U_q}^T \e) \e - \tau_q W_{U_q, V_q} \e
\end{equation}
for all $q=1,\dots,k$.
Rearranging \eqref{eq: lambda eq} shows that $\lambda_{U_q}$ is the solution to the system
\begin{equation} \label{eq: lambda mat vec eq}
	(m_q I + \e\e^T) \lambda_{U_q} = \tau_q W_{U_q, V_q} \e - \mu_1 \e
\end{equation}
for all $q=1,\dots,k$.
% Explicit formula for lambda.
As before, the Sherman-Morrision-Woodbury formula yields an explicit formula for $\lambda$;
for each $q \in \{1,\dots, k\}$,
applying \eqref{eq: SMW vector formula} with $A = m_q I$, $\u = \bv = \e$ shows that
\begin{equation}	\label{eq: choice of lambda}
	\lambda_{U_q} = \frac{1}{m_q} \left( \tau_q W_{U_q, V_q} \e - \frac{1}{2} \left( \mu_1 + \frac{\e^T W_{U_q,V_q} \e}{\tau_q n_q} \right) \e \right).
\end{equation}
%ensures that the rows of $S_{U_q, C_q}$ are orthogonal to the columns of $Z^*_{C_q, C_q}$.
Similarly, choosing
\begin{equation}	\label{eq: choice of phi}
	\phi_{V_q} = \frac{1}{n_q} \left( \frac{W_{U_q, V_q}^T \e}{\tau_{q}} - \frac{1}{2} \left(\mu_2 + \frac{\e^T W_{U_q,V_q} \e}{\tau_q n_q} \right) \e \right)
\end{equation}
forces the rows of $S_{V_q, C_q}$ to be orthogonal to the columns of $Z^*_{C_q, C_q}$ for all $q\in \{1,\dots,k\}$.
% Expectation of lambda, phi.
Note that 
%\begin{equation}	\label{eq: E[ lambda ] }
$	\E [\lambda_{U_q} ] %= \frac{1}{2 m_q} ( \alpha \tau_q n_q - \mu_1 ) \e
		%= \frac{1}{2} \left( \frac{\alpha}{\tau_q} - \frac{\mu_1}{m_q} \right) \e
		= (\alpha/(2\tau_q) - \mu_1/(2m_q) )\e$
%\end{equation}
for all $q \in \{1,\dots, k\}$.
We choose $\mu_1 = \epsilon \gamma \hat m$ 
for some scalar $\epsilon > 0$ to be defined later to ensure that  $\lambda$ is nonnegative in expectation.
Similarly,
%\begin{equation}	\label{eq: E[ phi ]}
$	\E [ \phi_{V_q} ] = \left( \alpha \tau_q /2-  \mu_2/(2n_q) \right) \e$
%\end{equation}
for all $q=1,\dots, k$.
Again,  we choose $\mu_2 = \epsilon \gamma  \hat n$ for small enough $\epsilon > 0$ to ensure that $\phi$ is nonnegative in expectation.

% Choice of eta.
We next construct the multiplier $\eta$.
We set $\eta_{C_{k+1}, C_{k+1}} = 0$ and
parametrize $\eta_{C_q, C_s}$ using the vectors $\y^{q,s}$ and $\z^{q,s}$
for each $q\neq s$.
For each $q = 1,\dots, k+1$, let $\w_q$ be the vector in $\R^{C_q}$ such that $\w_q(U_q) = \e$ and $\w_q(V_q) = \tau_q \e$.
We choose
% def of eta.
$%\begin{equation} \label{eq: eta formula}
	\eta_{C_q, C_s} = \Pi^{q,s} + \y^{q,s} \w_s^T + \w_q(\z^{q,s})^T
$%\end{equation}
, where
% def Pi.
$$
	\Pi^{q, s} = \mat{{cc} \pi_{U_q, U_s} \e\e^T & \tau_s \pi_{U_q, V_s} \e\e^T \\ \tau_q \pi_{V_q, U_s} \e\e^T & \tau_q \tau_s  \pi_{V_q, V_s} \e\e^T}
$$
for some scalars $\pi_{U_q,U_s}, \pi_{U_q, V_s}, \pi_{V_q, U_s}, \pi_{V_q, V_s} > 0$ to be defined later.
As before, we choose $\y^{q,s}$ and $\z^{q,s}$ to be solutions of the system of equations given by
$S_{C_q, C_s} Z^*_{C_s, C_s} = 0$ and $S_{C_s, C_q} Z^*_{C_q, C_q} = 0$.
By the symmetry of $S$ and $Z^*$, $\y^{q,s} = \z^{s,q}$ for all $q \neq s$.
%As in the previous section, this system of linear equations is a perturbation of a linear system with known solution, and we use
%the solution of the perturbed system to establish that $\eta$ is nonnegative and $S$ is positive semidefinte.

% Def b and Sbar
For all $q,s \in \{1,\dots,k+1\}$ such that $q\neq s$,  let
\begin{equation} \label{eq: Sbar def}
	\bar S_{C_q, C_s} :=	\mat{{cc}
							 \lambda_{U_q}\e^T + \e \lambda_{U_s}^T & - W_{U_q, V_s} \\
							 -W_{U_s, V_q}^T 						& \phi_{V_q} \e^T + \e \phi_{V_s}^T },
\end{equation}
and let $\b = \b^{q,s} \in \R^{C_q \cup C_s}$ be the vector defined by
$\b_{C_q} = \big(\bar S_{C_q,C_s}  -  \E[\bar S_{C_q, C_s}] \big)  \w_s$ and
$\b_{C_s} = \big( \bar S_{C_s, C_q} - \E[\bar S_{C_s, C_q}] \big) \w_q.$
%\begin{align}
%	\label{eq: b def q}
%	\b_{C_q} = \big(\bar S_{C_q,C_s}  -  E[\bar S_{C_q, C_s}] \big)  \w_s, \;\;\;
%	\b_{C_s} = \big( \bar S_{C_s, C_q} - E[\bar S_{C_s, C_q}] \big) \w_q.
%\end{align}
The parameters $\pi_{U_q,U_s}, \pi_{U_q, V_s}, \pi_{V_q, U_s}, \pi_{V_q, V_s} > 0$ will be chosen so that
% Row/Col sums of E - Pi = 0
\begin{equation}	\label{eq: c system1}
	\big(\E [\bar S_{C_q, C_s}] - \Pi ^{q,s} \big) \w_s= 0 \;\; \mbox{and} \;\;
%\end{equation}
%and
%\begin{equation}	\label{eq: c system2}
	\big( \E[\bar S_{C_s, C_q}] - \Pi^{s,q} \big) \w_q = 0.
\end{equation}
We will establish that such a choice of $\Pi^{q,s}$ exists in Lemma~\ref{lem: bound on  S2}.

Fix $q,s \in \{1,\dots, k\}$ such that $q\neq s$. It is easy to see that
the requirement that the rows of $S_{C_q, C_s}$ be orthogonal to the columns of $Z^*_{C_s, C_s}$ is satisfied if $\y=\y^{q,s}$ and $\z = \z^{q,s}$ are chosen to be
be the unique solutions of the system
\begin{equation} 	\label{eq: yz system t=1}
	\mat{{cc}
		2 m_s  I + \w_q \w_q^T 	&   0 \\
		0 & 2m_q I +  \w_s \w_s^T}
	{ \y \choose \z} 
	=
	\b.
\end{equation}
% Formula for y,z using SMW.
Applying \eqref{eq: SMW vector formula} with $A = 2 m_s$, $\u = \bv = \w_q$ and $A = 2 m_q$, $\u=\bv = \w_s$ yields
$$%\begin{equation} \label{eq: y formula}
	\y =  \frac{1}{2m_s} \rbra{ I - \frac{\w_q \w_q^T}{2(m_q + m_s) } } \b_{C_q}  \;\; \mbox{and} \;\;
	\z = \frac{1}{2 m_q} \rbra{ I - \frac{ \w_s \w_s^T}{2 (m_q + m_s) } } \b_{C_s}
$$%\end{equation}
respectively.

%% y,z in noise blocks.
For $q \in \{1,\dots, k\}$, we set $\z^{k+1, q } = 0$ and choose $\y = \y^{k+1,q}$  so that the rows of $S_{C_{k+1}, C_q}$ are orthogonal to $\w_q$.
By our choice of $\Pi^{k+1, q}$, $\y$ must satisfy
$$
	2 m_q \y 
		= \mat{{cc}	\e (\lambda_{U_q} - \E[\lambda_{U_q} ])^T & - W_{U_{k+1}, V_q} + \beta \e\e^T \\
					\beta \e\e^T - W_{U_q, V_{k+1}}^T & \e (\phi_{V_q} - \E[ \phi_{V_q} ] ) }
			\w_q
		= \b^{k+1, q} 
$$
Therefore, we choose 
%\begin{align}
	%\y^{k+1, q}_{U_{k+1}} &= \rbra{\frac{1}{2 m_q}}\b^{k+1,q}_{U_{k+1}} 	\label{eq: y noise def U}\\
	%\y^{k+1, q}_{V_{k+1}} &= \rbra{ \frac{1}{2 m_q } }\b^{k+1,q}_{V_{k+1}}	\label{eq: y noise def V}.
$	\y^{k+1, q} = (1/(2mq))\b^{k+1,q} 	\label{eq: y noise def V}.$
%\end{align}
We choose the remaining blocks of $\eta$ symmetrically. That is, we choose $\y^{q,k+1} = 0$ and set $\z^{q,k+1} = \y^{k+1, q}$ for all $q=1,\dots, k$.

%%%%%%%%%%%%%%%
% Summary of multipliers
%%%%%%%%%%%%%%%
%In summary, we choose the multipliers $\mu_1, \mu_2 \in \R$, $\lambda\in \R^{M}$, $\phi\in \R^{N}$, $\eta \in \R^{(M+N) \times (M+N)}$ as follows:
%\begin{align}
%	\mu_1 &= \epsilon \gamma\hat m \label{eq: mu1 choice} \\
%	\mu_2 &= \epsilon \gamma\hat n \label{eq: mu2 choice} \\
%	\lambda_{U_q} &=  
%		\branchdef{
%			\frac{1}{m_q} \left( \tau_q W_{U_q, V_q} \e - \frac{1}{2} \left( \mu_1 + \frac{\e^T W_{U_q,V_q} \e}{\tau_q n_q} \right) \e \right), &	q=1,\dots, k \\
%			0, & q = k+1}\label{eq: lambda choice} \\
%	\phi_{V_q} &= \branchdef{\frac{1}{n_q} \left( \frac{ W^T_{U_q,V_s} \e}{\tau_q} - \frac{1}{2} \left( \mu_2 + \frac{ \e^T W_{U_q, V_q} \e}{\tau_q n_q} \right) \e \right), &	q=1,\dots, k \\
%			0, & q = k+1}	 	\label{eq: phi choice} \\
%	\eta_{C_q, C_s} &=	\branchdef{ \Pi^{q,s} + \y^{q,s} \w_s^T + \w_q (\z^{q,s} )^T, &\mbox{if } q \neq s  \\ 0, & \mbox{otherwise,}	} \label{eq: eta choice}
%\end{align}
%where $\epsilon > 0$ is a scalar to be defined later, $\Pi^{q,s}$ is chosen so that \eqref{eq: c system1} is satisfied and $\y^{q,s}$ and $\z^{q,s}$ are given by \eqref{eq: y formula}, \eqref{eq: z formula},
%and \eqref{eq: y noise def V}
%for all $q \neq s$.
%We choose $S$ according to \eqref{eq: WKDB dual feas}. 
To establish that $S$ is positive semidefinite with high probability, we decompose $S$ as the sum $S = S_1 + S_2 + S_3 + S_4$ where
\begin{align}
	\label{eq: S1 def}
		S_1(C_q, C_s) &:= \branchdef{  		 S_{C_{k+1}, C_{k+1}}, & \mbox{if } q = s = k+1 \\
										\bar S_{C_q, C_s} - \E[\bar S_{C_q, C_s} ], & \mbox{otherwise,} }
	\\	\label{eq: S2 def}
		S_2(C_q,C_s) &:= \branchdef{ \E [ \bar S_{C_q, C_s} ] - \Pi^{q,s}, & \mbox{if } q \neq s  \\ \E[\bar S_{C_q,C_q} ], & \mbox{if } q = s, q \in \{1,\dots, k\}, \\ 0, & \mbox{otherwise,} } 
	\\	\label{eq: S3 def}
		S_3(C_q, C_s) &:= \branchdef{
							 \y^{q,s} \w_s^T + \w_q (\z^{q,s})^T  , &\mbox{for all } q,s \in \{1,\dots, k+1\}  
							 }
\end{align}						
and
\begin{equation}	\label{eq: S4 def}
		S_4 := \mat{{cc} \mu_1 I & 0 \\ 0 & \mu_2 I }.
\end{equation}
%Here, and in the remainder of the paper, the notation $M(A,B)$ refers to the $|A|\times |B|$ submatrix of the matrix $M$ with rows indexed by the set $A$ and columns indexed by the set $B$, e.g.,
%$S_1(C_q, C_s) = [S_1]_{C_q, C_s}$.
%Similarly, the notation $\bv(A)$ will refer to the vector in $\R^{|A|}$ with entries corresponding to those of $\bv$ indexed by the set $A$.

We conclude with the following theorem, which provides a sufficient condition for optimality and uniqueness of the proposed solution $Z^*$ for \eqref{eq: WKDB relaxation}.
%-------------------------------------------------------------------------
% Sufficient condition for opt and uniqueness.			
%-------------------------------------------------------------------------

\begin{theorem} \label{thm: WKDB sufficient}
%	Suppose that the bicliques $(U_1, V_1)$, $(U_2, V_2)$, \dots, $(U_k,V_k)$ form a \kdbs $G^*$ of the bipartite complete graph $K_{M,N} = ((U,V),E)$.	
%	Let $U_{k+1} := U \setminus (\cup_{i=1}^k U_i)$ and $V_{k+1} := V \setminus (\cup_{i=1}^k V_i)$.
%	Let $m_i := |U_i|$ and $n_i := |V_i|$ for all $i=1,\dots,k+1$.
%	Let $W \in \R^{M\times N}$ be a random matrix sampled from the planted bicluster model
%	according to distributions $\Omega_1, \Omega_2$ with means $\alpha, \beta$ satisfying \eqref{eq: a twice b}.
%	Suppose that $m_i = \tau_i^2 n_i$ for all $i \in \{1,\dots, k +1\}$ such that the scalars $\{\tau_1, \dots,  \tau_{k+1}\}$
%	satisfy \eqref{eq: abd assumption}
%	for all $i,j \in \{1,\dots, k+1\}$.
	Let $Z^*$ be the feasible solution for \eqref{eq: WKDB relaxation} corresponding to $G^*$ defined by \eqref{eq: WKDB sol}.
	Then there exist scalars $\xi_1, \xi_2 > 0$ such that if
	\begin{equation} \label{eq: S1 bound cond}
		\|S_1\| + \xi_1 (n_{k+1} N)^{1/2}\le \xi_2 \gamma \hat n
	\end{equation}
%	then $Z^*$ is the optimal for \eqref{eq: WKDB relaxation}
%	and $G^*$ is the densest \kdbs of $K_{M,N}$ corresponding to $W$ with \pteto{\hat n}.
%	Moreover, if
%	\begin{equation}	\label{eq: bicluster uniqueness cond}
%		n_s \e^T W_{U_q, V_q} \e > n_q \e^T W_{U_q, V_s} \e
%	\end{equation}
%	for all $q,s \in \{1,\dots, k\}$ such that $q\neq s$, 
	then $Z^*$ is the unique optimal solution of \eqref{eq: WKDB relaxation}
	and $G^*$ is the unique densest \kdbs of $K_{M,N}$ with \pteto{\hat n}.

\end{theorem}

%-------------------------------------------------------------------------
% Start of proof.
%-------------------------------------------------------------------------
The remainder of this section consists of a proof of Theorem~\ref{thm: WKDB sufficient}.
We first establish that $Z^*$ is optimal for \eqref{eq: WKDB relaxation} and $G^*$ is the unique densest \kdbs of $K_{M,N}$ with \pteto{\hat n}
if \eqref{eq: S1 bound cond} is satisfied.
By construction, $\mu, \lambda,\phi, \eta$ and $S$ satisfy \eqref{eq: WKDB dual feas}, \eqref{eq: WKDB CS U rowsum}, \eqref{eq: WKDB CS V rowsum},
\eqref{eq: WKDB CS nonneg}, and \eqref{eq: WKDB CS sdp}. 
Moreover, a series of arguments similar to those in Section~\ref{sec: c nonneg} establish that $\lambda,\phi$, and $\eta$ are nonnegative
with \pteto{\hat n}.
Therefore, it suffices to show that $S$ is positive semidefinite with \pteto{\hat n} if  \eqref{eq: S1 bound cond} is satisfied.
To do so, we will establish that $\x^T S \x \ge 0$ for all $\x \in \R^{M+N}$ in this case.

Fix $\x \in \R^{M+N}$. We decompose $\x$ as $\x = \sum_{i=1}^k \varphi_i \x_i + \bar\x$ 
for some $\varphi_1, \dots, \varphi_k$, where
$ \x_i (C_i) = \w_i$ and all remaining entries of $\x_i$ are equal to $0$,
and $\bar\x$ is orthogonal to the span of $\{\x_1, \dots, \x_k\}$.
Note that $\spn\{\x_1, \dots, \x_k\} \subseteq \Null{S}$
since $\x_i$ is a scalar multiple of a column of $Z^*$ for all $i=1,\dots, k$.
It follows that
\begin{equation}	\label{eq: xSx expansion}
	\x^T S \x = \bar\x^T S \bar\x = \sum_{i=1}^{4} \bar\x^T  S_i \bar\x.
\end{equation}
Note that $\bar\x^T S_3 \bar\x = 0$ since $\bar\x(C_q)$ is orthogonal to $\w_q$ for all $q = 1, \dots, k$ and
% Lower bound on x' S4 x.
$%\begin{equation} \label{eq: x S4 x lower bound}
	\bar\x^T  S_4 \bar\x \ge \min \{\mu_1, \mu_2\} \| \bar\x\|^2
		= \epsilon \gamma\min \{\hat m, \hat n \} \| \bar\x \|^2.
$
by our choice of $\mu_1$ and $\mu_2$.
The following lemma provides a similar lower bound on $\bar\x^T  S_2 \bar\x$.

%-------------------------------------------------------------------------
% Bound on S2.
%-------------------------------------------------------------------------

\begin{lemma}	\label{lem: bound on  S2}
	Suppose that $\alpha, \beta, \tau_1, \dots, \tau_{k+1}$ satisfy \eqref{eq: a twice b} and \eqref{eq: abd assumption}.
	Then, for all $q,s \in \{1,\dots, k + 1\}$ such that $q \neq s$,
	there exist scalars $\pi_{U_q,U_s}, \pi_{U_q, V_s}, \pi_{V_q, U_s}, \pi_{V_q, V_s} > 0$ and $\hat c>0$,
	depending only on $\alpha, \beta, \tau_1,\dots, \tau_{k+1}$ such that 
	$%\begin{equation}	\label{eq: nonneg x S2 x}
		\bar\x^T S_2 \bar\x \ge -  \hat c\|\bar\x \|^2\sqrt{n_{k+1} N} 
	$%\end{equation}
	and \eqref{eq: c system1} is satisfied.
%	\begin{align}
%		\label{eq: S2 row sums = 0}
%			\left( E [ \bar S_{C_q, C_s}] - \Pi^{q,s} \right)\w_s = 0, \;\;\;
%%		\label{eq: S2 col sums = 0}
%			\left( E [ \bar S_{C_s, C_q}] - \Pi^{s,q}\right) \w_q = 0.	
%	\end{align}
\end{lemma}

\begin{proof}
	% Formula for pi's
	Fix $q, s \in \{1,\dots, k\}$ such that $q\neq s$.
	Let $\pi_1 :=\pi_{U_q,U_s},$ $\pi_2:= \pi_{U_q,V_s}$, $\pi_3 := \pi_{V_q, U_s}$, and $\pi_4 :=\pi_{V_q, V_s}$.
	Then the system of equations defined by \eqref{eq: c system1}  % and \eqref{eq: S2 col sums = 0}
	is equivalent to 
	\begin{equation}	\label{eq: c system}
		\mat{{cccc}	1	& 1 	& 0 	& 0 \\
				0	& 0	& 1 	& 1 \\
				1	& 0	&1 	& 0 \\
				0	& 1	& 0	& 1 }
		\vect{ \pi_1 \\ \pi_2 \\ \pi_3 \\ \pi_4 }
		= 
		\vect{ \bar \lambda - \beta/ \tau_s \\ \bar \phi - \tau_s \beta \\ \bar \lambda - \beta/\tau_q \\ \bar \phi - \tau_q \beta},
	\end{equation}
	where
	\begin{equation}	\label{eq: b lambda, b phi def}
		\bar \lambda := \frac{\alpha}{2} \left( \frac{1}{\tau_q} + \frac{1}{\tau_s}\right) - \frac{\mu_1}{2} \left( \frac{1}{m_q} + \frac{1}{m_s} \right) ,\;\;\;\;\;
		\bar \phi := \frac{\alpha}{2} ( \tau_q + \tau_s) - \frac{\mu_2}{2} \left( \frac{1}{n_q} + \frac{1}{n_s} \right).
	\end{equation}
	The system \eqref{eq: c system} is singular with solutions of the form
	\begin{align} \label{eq: c1 formula} 
			\pi_1 &= \bar \lambda - \frac{ \bar \phi}{\tau_q \tau_s} + \pi_4, \;\;\;\;			 
			\pi_2 = 	\frac{ \bar \phi}{\tau_q \tau_s}  - \frac{\beta}{\tau_s} - \pi_4, \;\;\;\;
			\pi_3 = \frac{ \bar \phi}{\tau_q \tau_s}  - \frac{\beta}{\tau_q} - \pi_4.
	\end{align}
	% Nonnegativity of pi's
	We next show that there exists some choice of $\pi_4 > 0$, independent of $\hat n$, such that
	the desired bound on $\bar\x^T S_2 \bar\x$  holds and	$\pi_1,\pi_2, \pi_3$ are bounded below by a positive scalar whenever  \eqref{eq: a twice b} and \eqref{eq: abd assumption} are satisfied.
	
	Suppose that $\alpha,$ $\beta$, $\tau_1, \dots, \tau_{k+1}$ satisfy \eqref{eq: a twice b} and \eqref{eq: abd assumption}.
	Let $\pi_4 := ( \rho_1 \bar\phi - \rho_2 \beta ) /\tau_q \tau_s$
	for some $\rho_1, \rho_2 > 0$ to be chosen later.
	For $\pi_4$ to be strictly positive, we need 
	$
		\rho_2 \beta < \rho_1 \bar \phi.
	$
	Substituting our choice of $\pi_4$ into the formulas for $\pi_2$ and $\pi_3$  given by \eqref{eq: c1 formula}
	and rearranging shows that $\rho_1$ and $\rho_2$ must satisfy
	\begin{equation}	\label{eq: p2 B lower bound}
		\rho_2 \beta > \beta \max\{\tau_q, \tau_s \} + (\rho_1 - 1) \bar \phi
	\end{equation}
	for $\pi_2, \pi_3$ to be positive.
	When \eqref{eq: abd assumption} is satisfied
	$$
		\bar \phi - \beta \max\{\tau_q, \tau_s\}  \ge \frac{\alpha}{2}(\tau_q + \tau_s)  - \beta\max\{\tau_q,\tau_s\} - \epsilon \gamma  > 0
	$$
	for sufficiently small $\epsilon > 0$ in  our choice of $\mu_1$ and $\mu_2$.
	Therefore, we choose $\rho_2$ such that
	$$
		\rho_2 = \rho_1 \bar \phi - \kappa (\beta \max \{\tau_q, \tau_s\} - \bar \phi\}
	$$
	for some $\kappa \in (0,1)$.
	Then
	$
	\pi_4 = \kappa (\bar \phi - \beta \max\{\tau_q, \tau_s \} )
	$
	is bounded below by a positive scalar, depending only on $\alpha, \beta, \tau_q,$ and $\tau_s$ by our choice of $\mu_2$.
	Since our choice of $\rho_1,\rho_2$ satisfies \eqref{eq: p2 B lower bound}, $\pi_2, \pi_3$
	are also bounded below by a positive scalar.
	Finally, since $\pi_4$ is at least a positive scalar,  we can always take $\epsilon > 0$ small enough
	that  $\pi_1$ is also
	bounded below by a positive scalar depending only on $\alpha, \beta, \tau_q$ and $\tau_s$.
	The case when $q \in \{1,\dots, k\}$ and $s = k+1$ follows by an identical argument.

	It remains to show that this particular choice of $\Pi$ yields the desired lower bound on $\bar \x^T S_2 \bar \x$.
	Let $\u_q:= \bar\x(U_q)$ and  $\bv_q := \bar\x(V_q)$ denote the entries of $\bar\x$ indexed by $U_q$ and $V_q$ respectively, for all $q=1,\dots , k+1$.
	For all $q=1,\dots, k$, we have
	$
		\u_q^T \e = - \tau_q \bv_q^T \e
	$
	since $\bar\x$ is orthogonal to $\spn\{\x_1, \dots, \x_k\}$.
	% Formula for noise blocks.
	Fix $s \in \{1,\dots, k\}$.
	By our choice of $\pi_1^{k+1,s},$ $\pi_2^{k+1,s},$ $\pi_3^{k+1,s},$ and $\pi_4^{k+1,s}$ we have
	\begin{align*}
		S_2(C_{k+1}, C_s) = S_2(C_s, C_{k+1})^T  &=
			\frac{1}{2} \mat{{cc} (\bar\lambda_{k+1, s} + \beta/\tau_s) \e\e^T & - (\tau_s \bar \lambda_{k+1,s} + \beta) \e\e^T \\
							- (\bar \phi_{k+1, s}/\tau_s + \beta ) \e\e^T &  ( \bar \phi_{k+1,s} + \tau_s \beta ) \e\e^T }  \\
		& = \frac{1}{2}  {(\bar \lambda_{k+1,s} +  \beta/\tau_s )\e \choose -(\bar \phi_{k+1, s} /\tau_s + \beta ) \e} (\e^T \; - \tau_s\e^T) .
	\end{align*}
	It follows that
	\begin{align*}
		\sum_{s=1}^k &\bar\x(C_{k+1})^T S_2 (C_{k+1}, C_s) \bar\x (C_s) \\			
			&\ge - \frac{1}{2}\sum_{s=1}^k \max \bbra{ \bar\lambda + \frac{\beta}{\tau_s}, \frac{\bar \phi}{\tau_s} + \beta } \|\bar\x(C_{k+1})  \|_1 \rbra{ \|\u_s\|_1 + \tau_s \|\bv_s\|_1 } \\
			%&\ge - \frac{1}{2}\rbra{ \frac{\alpha}{2} + \beta } \rbra{\frac{\max\{\tau_{\max}, 1\}} { \min\{\tau_{\min} , 1\} } }  \sum_{s=1}^k \|\bar\x(C_{k+1})\|_1 \|\bar\x(C_s)\|_1 \\ 
			&\ge= -\hat c \|\bar\x(C_{k+1})\|_1 \rbra{ \|\bar\x \|_1 - \|\bar\x(C_{k+1})\|_1 },
	\end{align*}
	where $\tau_{\min} := \min_{i=1,\dots,k } \tau_i$, $\tau_{\max} := \max_{i =1 ,\dots, k} \tau_i$, and
	$$
		\hat c :=  \frac{1}{2}\rbra{ \frac{\alpha}{2} + \beta } \rbra{ \frac{\max\{\tau_{\max}, 1\}} { \min\{\tau_{\min} , 1\} }  }.
	$$
	The optimization problem
	$$
		\max_{\w_1 \in \R^{\ell_1}, \w_2 \in \R^{\ell_2}} \big\{ \|\w_1\|_1 \|\w_2\|_1 : \|\w_1\|^2 + \|\w_2\|^2 = \Psi^2 \big\}
	$$
	has optimal solution $\w_1^* = (\Psi/\sqrt{2\ell_1}) \e$, $\w_2^* = (\Psi/\sqrt{2\ell_2}) \e$, with optimal  value 
	$ \Psi^2 \sqrt{\ell_1 \ell_2} /2$.
	Taking $\w_1 := \bar\x(C_{k+1})$ and $\w_2 = (\bar\x(C_1) ; \dots ; \bar\x(C_k))$ and $\Psi = \|\bar\x\|$, shows that
	$$
		\|\bar\x(C_{k+1})\|_1 \big( \|\bar\x \|_1 - \|\bar\x(C_{k+1})\|_1 \big) \le \frac{\|\bar\x \|^2}{2}\sqrt{r_{k+1} (N - r_{k+1} ) } 
	$$
	and, consequently,
	\begin{equation}	\label{eq: S2 noise blocks}
		\sum_{s=1}^k \bar\x(C_{k+1})^T S_2 (C_{k+1}, C_s) \bar\x (C_s)  \ge  - \frac{\hat c \|\bar\x \|^2}{2}\sqrt{r_{k+1} N} .
	\end{equation}		
	% Formula for diagonal blocks
	Similarly,	
	\begin{align}
		\bar\x&(C_q)^T  S_2(C_q, C_q) \bar\x(C_q)  = (\bv_q^T\e)^2 \left( 4 \tau_q \alpha - \frac{ \mu_1 + \mu_2}{n_q }\right) \label{eq: x S2 x diag block}
%		& = {\u_q \choose \bv_q} ^T
%				\mat{{cc}	(\alpha/\tau_q - \mu_1/m_q)\e\e^T 	& -\alpha \e\e^T \\
%						- \alpha \e\e^T 					& (\alpha \tau_q - \mu_q / n_q ) \e\e^T }				
%				  {\u_q \choose \bv_q} \notag \\
%		& = (\u_q^T \e)^2 \left( \frac{\alpha}{\tau_q} - \frac{\mu_1}{m_q} \right) - 2 (\u_q^T \e)(\bv_q^T \e) \alpha
%			+ (\bv_q^T \e)^2 \left( \alpha \tau_q - \frac{\mu_2}{n_q} \right) 	\notag\\
%		&
	\end{align}
	for all $q = 1, \dots, k$.
	% Formula for off-diagonal blocks
	Finally, for $q,s \in \{1,\dots, k\}$ such that $q \neq s$, we have
	\begin{align}
		&\bar\x(C_q)^T  S_2(C_q, C_s) \bar\x(C_s)  \notag \\
%			&={\u_q \choose \bv_q}^T 
%				\mat{{cc}	(\bar \lambda^{q,s} - \pi_1)\e\e^T 	& - (\beta +  \tau_s \pi_2^{q,s}) \e\e^T \\
%						- (\beta + \tau_q\pi_3^{q,s})\e\e^T 					& (\bar\phi^{q,s} -  \tau_q \tau_s\pi_4^{q,s} ) \e\e^T }				
%				  {\u_s \choose \bv_s} \notag \\
			&= (\bv_q^T \e)(\bv_s^T\e) \Big( \tau_q \tau_s \bar\lambda^{q,s} + \beta ( \tau_q + \tau_s ) + \bar \phi^{q,s} - \tau_q \tau_s (\pi_1^{q,s} - \pi_2^{q,s} - \pi_3^{q,s} + \pi_4^{q,s} ) \Big)	\label{eq: x S2 x off-diag block 1} \\
			&=  4 (\bv_q^T \e)(\bv_s^T\e) (\bar \phi^{q,s} - \tau_q \tau_s \pi_4^{q,s} ).	\label{eq: x S2 x off-diag block 2}
	\end{align}
	Here \eqref{eq: x S2 x off-diag block 2} is obtained by substituting \eqref{eq: c1 formula},
	into \eqref{eq: x S2 x off-diag block 1} .
	Let  $\bar v_q := \bv_q^T \e$ for all $q=1,\dots, k$.
	Combining  \eqref{eq: S2 noise blocks}, \eqref{eq: x S2 x diag block} and \eqref{eq: x S2 x off-diag block 2} shows that
	% Lower bound on x' S2 x.
	\begin{align*}
		&\bar\x^T   S_2 \bar\x  \\
			&\ge - \hat c\|\bar\x \|^2\sqrt{r_{k+1} N} + \sum_{q=1}^k  \bar v_q^2 \rbra{ \tau_q \alpha - \frac{\mu_1 + \mu_2}{n_q} }+
				 2 \sum_{q=1}^k \sum_{s = q+1}^k 4 \bar v_q \bar v_s ((1-\kappa)\bar\phi^{q,s} + \kappa \beta \tilde \tau_{qs} ) \\
			 &\ge  - \hat c\|\bar\x \|^2\sqrt{r_{k+1} N}  + 8 \sum_{q=1}^k \sum_{s = q+1}^k  |\bar v_q  \bar v_s |
					\rbra{\alpha \tau_{\min}  - \frac{\mu_1 + \mu_2}{4\hat n} - (1-\kappa) \bar\phi^{q,s} - \kappa \beta \tilde \tau_{qs}},
	\end{align*}
	where $\tilde \tau_{qs} := \max \{\tau_q, \tau_s\}$, since $\sum_{q=1}^k \bar v_q^2 \ge \sum_{q\neq s} |\bar v_q \bar v_s |$.	
	% Proof that this lower bound is nonnegative.
	If $\alpha \tau_{\min} > \beta \tau_i$ for all $i=1,\dots, k$ then, for all $\epsilon > 0$ sufficiently small and $\kappa$ sufficiently close to $1$, we have
	\begin{align*}
		\alpha \tau_{\min} & - \frac{\mu_1 + \mu_2}{4\hat n} - (1-\kappa) \bar\phi^{q,s} - \kappa \beta \max \{\tau_q, \tau_s\}  \\
		& \ge \alpha \tau_{\min} - \beta \max\{\tau_q, \tau_s\} - (1-\kappa)(\alpha-\beta) \max\{\tau_q,\tau_s\}
			- \frac{\epsilon \gamma  }{4} \left( \frac{\hat m}{\hat n} + 1 \right) \ge 0		
	\end{align*}
	for all $q\neq s$.
	It follows immediately that $\bar\x^T  S_2 \bar\x \ge - \hat c\|\bar\x \|^2\sqrt{r_{k+1} N} .$ 
	\qed
	\bigskip
\end{proof}

%---------------------------------------------------------------------------------------------------------------------------------
% Putting all the bounds together.
%---------------------------------------------------------------------------------------------------------------------------------

Substituting the respective bounds on $\bar \x^T S_i \bar \x$ into \eqref{eq: xSx expansion} shows that
\begin{equation} \label{eq: xSx lower bound}
	\bar\x^T S \bar\x \ge \left(\min\{\mu_1, \mu_2\} - \hat c \sqrt{r_{k+1}  N }- \|S_1\| \right) \|\bar\x\|^2.
\end{equation}
Since $\mu_1, \mu_2$ are both  scalar multiples of $\hat n$, where the scalar depends only on $\alpha, \beta, \tau_1,\dots, \tau_{k+1}$, there exists
scalar $\xi > 0$, also depending only $\alpha, \beta, \tau_1,\dots, \tau_{k+1}$, such that 
the right-hand side of \eqref{eq: xSx lower bound} is nonnegative 
if  $\| S_1 \| + \hat c \sqrt{r_{k+1}  N } \le \xi \gamma \hat n$.

%---------------------------------------------------------------------------------------------------------------------------------
% Uniqueness bound
%---------------------------------------------------------------------------------------------------------------------------------
It remains to show that $Z^*$ is the unique optimal solution with high probability if \eqref{eq: S1 bound cond} holds.
An argument similar to that in the proof of Theorem~\ref{T: S opt conds} show that $Z^*$ is the unique optimal solution of \eqref{eq: WKDB relaxation}
if $n_s \e^T W_{U_q, V_q} \e > n_q \e^T W_{U_q, V_S} \e$
for all $q,s \in \{1,\dots, k\}$.
Moreover, an argument identical to that of the proof of Lemma \ref{T: WKDC uniqueness},
establishes that this uniqueness condition holds with high probability for sufficiently large $\hat n$.
This completes the proof.

%\subsection{Nonnegativity of the dual variables}
%\label{sec: b nonneg}
%\input{b-nonneg-proof.tex}

\subsection{Positive semidefiniteness of $S$}
\label{sec: w S bound}
% Exposition.
It remains to show that $S$, as defined by \eqref{eq: WKDB dual feas}, satisfies \eqref{eq: S1 bound cond} to prove that $Z^*$ is the unique optimal solution of \eqref{eq: WKDB relaxation}.
In particular, we will derive the following upper bound on the spectral norm of $S_1$.

% S bound thm.
\begin{lemma} \label{thm: B S bound}	
	There exist scalars $c_1, c_2 > 0$ such that
	\begin{equation} \label{eqn: B S bound}
		\| S_1\| \le  c_1\sqrt{ k N }
			+ c_2 \sqrt{N}  + \beta \tau_{k+1} n_{k+1}
	\end{equation}
	with probability tending exponentially to $1$ as $\hat n$ approaches $\infty$.	
\end{lemma}

%This lemma, along with 
%Theorem~\ref{thm: WKDB sufficient}  %and Lemmas~\ref{lem: lower bound on lambda, phi},~\ref{lem: bicluster y,z inf bounds},~and~\ref{thm: unique cond holds}
%establishes Theorem~\ref{thm: biclustering recovery guarantee}.
%Indeed, if the right-hand side of \eqref{eqn: B S bound} is at most $\Omega \rbra{ \gamma\hat n} -  O ((N n_{k+1})^{1/2})$ 
%then Theorem~\ref{thm: WKDB sufficient}  implies that the planted \kdbs is the maximum density \kdbs of $K_{M,N}$ with respect to $W$
%and can be recovered by solving \eqref{eq: WKDB relaxation}.
%-----------------------------------------------------------------
% tilde S decomp.
%-----------------------------------------------------------------
To establish Lemma~\ref{thm: B S bound},
%S1
we decompose $S_1$ as
$
     S_1 = \tilde S_1 + \tilde S_2 + \tilde S_3 +  \tilde S_4,
$
where $\tilde S_i \in \Sigma^{M+N}$, $i=1,\dots, 4$, are defined as follows.
We take
\begin{align*}
     \tilde S_1(U_q, U_s)  &= { (\lambda_{U_q} - \E[\lambda_{U_q}])\e^T  +  \e ( \lambda_{U_s} - \E[\lambda_{U_s} ] )^T ,  }  \\
     \tilde S_1(V_q, V_s) & = {  (\phi_{V_q} - \E[\phi_{V_q}])\e^T + \e ( \phi_{V_s} - \E[\phi_{V_s} ] )^T, }  
\end{align*}
for all $q, s \in \{1,\dots, k+1\}$
and set all remaining entries of $\tilde S_1$ to be $0$.
%S2
Next,
let
$$
	\tilde S_2(U_q, V_s) = \branchdef{ \beta \e\e^T - R^{q,q}, &\mbox{if } q = s, \; q \in \{1,\dots,k\} \\ \beta \e\e^T - W_{U_q, V_s}, &\mbox{otherwise},}
$$
where $R^{q,q}$ is a $m_q \times n_q$ random matrix with \iid entries sampled according to $\Omega_2$, the distribution of the off-diagonal blocks of $W$.
We choose $\tilde S_2(V_q, U_s) =  \tilde S_2(U_s, V_q)^T$ and set all other entries of $\tilde S_2$ equal to $0$.
% S3.
Next, we set 
$
	\tilde S_3(U_q, V_q) = \alpha \e\e^T - W_{U_q, V_q}
$
and $\tilde S_3(V_q, U_q) = \tilde S_3(U_q, V_q)^T$
for all $q = 1,\dots, k$, and set all remaining entries of $\tilde S_3$ equal to $0$.
% S4.
Finally, $\tilde S_4$ is the correction matrix for the  diagonal blocks of $\tilde S_2$. That is, we take
$\tilde S_4 (U_q, V_q) = R^{q,q} - \beta \e\e^T$, and $\tilde S_4(V_q, U_q) = \tilde S_4(U_q, V_q)^T$,
for all $q=1,\dots, k$, we take
$
	\tilde S_4(U_{k+1}, V_{k+1}) = \tilde S_4(V_{k+1}, U_{k+1})^T = - \beta \e\e^T,
$
and all remaining entries of $\tilde S_4$ are $0$.
To obtain the desired bound on $\|S_1\|$, we bound each of $\|\tilde S_1\|$, $\|\tilde S_2\|$, $\|\tilde S_3\|$, and $\|\tilde S_4\|$ individually.
To do so, we will repeatedly invoke the following bound on the norm of a random rectangular matrix
(see \cite{Geman:1980} and \cite[Theorem~2]{Ames-Vavasis}).

\begin{theorem}
    \label{geman tail bound}
    Let $A$ be a $\lceil yn \rceil \times n$ random matrix with \iid entries sampled
    from a  distribution with mean $\mu$ and variance $\sigma^2$ such that 
    $A_{ij} \in [0,1]$ for all $i \in \{1,\dots, \ceil{yn}\}$, $j \in \{1, \dots, n\}$
    for fixed $y \in \R_+$.
    Then there exist $c_1, c_2, c_3, c_4 > 0$ depending only on $\sigma$ and $y$ such that
    $
        \|A - \mu \e\e^T \| \le c_4 \sigma \sqrt{n}
    $
    with probability at least $1 - c_1 \exp(-c_2 n^{c_3})$.
\end{theorem}

% Geman bounds on S2-S4
Applying Theorem~\ref{geman tail bound}  shows that
$
	\|\tilde S_2 \| = \|\tilde S_2 (U,V) \| \le \tilde c_2 \sqrt{N}
$
for some scalar $\tilde c_2$ with \pteto{\hat n} by the block structure of $\tilde S_2$.
A similar argument shows that there exists scalar $\tilde c_3>0$ such that
$%\begin{equation} \label{eq: S3 bound}
	\|\tilde S_3 \| %= \| \tilde S_3(U,V) \| 
	\le \tilde c_3 \max_{q=1,\dots, k} \sqrt{n_q} \le \tilde c_3 \sqrt{N}
$%\end{equation}
with \pteto{\hat n}.
Next,
$$
	\|\tilde S_4 \| =  \max \left\{ \tilde c_4 \max_{q=1,\dots, k} \sqrt{n_q}, \beta \sqrt{m_{k+1} n_{k+1}} \right\}  \le \tilde c_4 \sqrt{N} + \beta \tau_{k+1} n_{k+1}
$$
for some scalar $\tilde c_4 > 0$ with \pteto{\hat n}, again by Theorem~\ref{geman tail bound}.
Finally, a calculation similar to the derivation of the bound on $\|\tilde S_2\|$ in the proof of Lemma~\ref{thm: S bound}	
shows that $\|\tilde S_1\| = O(\sqrt{k N})$ with \pteto{\hat n}. Applying the triangle inequality and the union bound completes the proof.
\section{Numerical Experiments}
\label{sec: expts}
%-----------------------------------------------------------------------------------------------
% Section outline.
%-----------------------------------------------------------------------------------------------
In this section, we empirically verify the performance of our heuristics for a variety of program inputs.
Specifically, we randomly generate symmetric matrices $W$ according to the planted cluster model, for
a number of distributions on the entries of $W$ and partitions $\{C_1, \dots, C_{k+1}\}$ of the rows and columns of $W$,
and compare the optimal solution of \eqref{E: w trace relax} to that corresponding to the planted partition.
Similarly, we also compare the optimal solution of \eqref{eq: WKDB relaxation}  to the matrix representation
of the planted partition  for $W$ sampled from the planted bicluster model.

%-----------------------------------------------------------------------------------------------
% Description of the algorithm.
%-----------------------------------------------------------------------------------------------
In each experiment, we solve either \eqref{E: w trace relax} or \eqref{eq: WKDB relaxation} using the Alternating Direction Method of Multipliers (ADMM).
A comprehensive description of ADMM and similar algorithms is well beyond the scope of this manuscript; we direct the reader to the recent survey \cite{boyd2011distributed} for more details.
To solve \eqref{E: w trace relax}, we represent the feasible region as the intersection of two sets and apply ADMM to solve the resulting equivalent formulation.
In particular, let
$\Xi := \{X \in \Sigma^V: X\e \le \e, \; X \ge 0\},$
and 
$\Omega:=\{X \in \Sigma^V: \tr(X) = k, \; X \in \Sigma^V_{+} \}.$
Then we may rewrite \eqref{E: w trace relax} as
$
	\max \{ \tr(WY): X - Y = 0,\; X \in \Xi, \; 	Y \in \Omega \}.
$
We solve this problem iteratively as follows. In each iteration, we approximately minimize the augmented Lagrangian 
$
	L_\beta (X,Y,U) = \tr(WY) - \tr(U(X-Y)) + \frac{\beta}{2} \|X - Y\|^2_F
$
with respect to $Y$ and $X$ successively, and then update the dual multiplier $U$ as $U = U - \beta(X-Y)$.\footnote{The penalty parameter $\beta = \min \{ \max \{5n/k, 80 \}, 500\}/2$ was chosen via simulation, and seems to work well
for most problem instances.}
Here $\|\cdot\|_F$ denotes the Frobenius norm on $\Sigma^V$ defined by $\|X\|_F^2 = \tr(X^2)$.
As we will see, the resulting subproblems are convex and can be solved efficiently; therefore, this algorithm will converge to the optimal solution
of \eqref{E: w trace relax} (see \cite[Theorem~8]{eckstein1992douglas}).

Let $(X^k, Y^k, U^k)$ be the current iterate after $k$ iterations.
% Update Y.
To update in the $Y$ direction, we minimize $L_\beta$ with respect to $Y$. That is, $Y^{k+1}$ is a minimizer of the subproblem
\begin{equation} \label{eq: y subproblem}
	\min_{Y \in \Omega} \frac{1}{2} \left\| Y - \rbra{ X^k - \frac{W + U^k}{\beta} }  \right\|^2_F.
\end{equation}
Let $X^k - (W+U^k)/\beta$	have eigenvalue decomposition $V \Diag (\bv^k) V^T$. Then, by the fact that both the Frobenius norm and the set $\Omega$ are invariant under 
unitary similarity transformations, we have
$
	Y^{k+1} = V \Diag (\y^*) V^T,
$	
where $\y^*$ is the optimal solution of
$
	\min \{ \| \y - \bv^k \|^2: \e^T \y = k, \; \y \ge 0 \},	
$	
by \cite[Proposition 2.6]{lu2010penalty}.
This latter subproblem admits an analytic solution, which can be computed efficiently; see \cite{van2008probing}.

% Update X.
Next, we take $X^{k+1}$ to be the optimal solution of
\begin{equation} \label{eq: x subproblem}
	\min_{X \in \Xi} \frac{1}{2} \left\|X - \rbra{ Y^k + \frac{U^k}{\beta} } \right\|^2_F.
\end{equation}
Unfortunately, this subproblem does not admit a closed-form solution. Instead, we approximately solve \eqref{eq: x subproblem}
by applying the spectral projected gradient method of \cite{birgin2000nonmonotone} to the dual of \eqref{eq: x subproblem}.
Taking the dual  of \eqref{eq: x subproblem} shows that
$$
	X^{k+1} = \left[  \rbra{Y^{k+1} + \frac{U^k}{\beta} } - \frac{\z^* \e^T + \e (\z^*)^T}{2} \right]_+,
$$
where $\z^*$ is the optimal solution of the dual problem
\begin{equation} \label{eq: x dual}
	\min_{\z \ge 0} \frac{1}{2} \left\| \left[ 	\rbra{Y^{k+1} + \frac{U^k}{\beta} } - \frac{\z \e^T + \e \z^T}{2} \right]_+ \right\|_F^2  + \z^T \e - \frac{1}{2} \left\| Y^{k+1} + \frac{U^k}{\beta} \right\|^2_F;
\end{equation}
Here, the operator $[\cdot]_+: \Sigma^V \ra  \Sigma^V \cap \R^{V\times V}_+$ maps  each $Z \in \Sigma^V$ to the matrix with $(i,j)$th  entry equal to 
$
	[[Z]_+]_{ij} = \max \{0, Z_{ij} \}
$	
for all $i,j \in V$.
Moreover, the objective function of the dual problem \eqref{eq: x dual} is both differentiable and coercive in $\z$, and, therefore, the dual
can be solved efficiently \cite{birgin2000nonmonotone}.
% Stopping criteria.
The algorithm is stopped when the relative duality gap $|v_p^{(k)} - v_d^{(k)}|/\max\{|v_p^{(k)}, 1\}$ 
and primal constraint violation are smaller than a desired error tolerance.

%% WKDB.
We solve \eqref{eq: WKDB relaxation} in a similar manner. In particular, we apply ADMM to minimize the augmented Lagrangian of the convex program
$
	\max \{ \tr(WY): X - Y = 0, \; X \in \Xi_B,\; Y \in \Omega_B\},
$
where $\Xi_B = \{X \in\Sigma^{U \cup V}: X_{U,U} \e \le \e, \; X_{V,V}\e \le \e, \; X \ge 0\}$ and 	$\Omega_B = \{Y \in \Sigma^{U\cup V}_{+}: \tr(Y) = 2k \}.$
It is important to note that $\Omega_B$ is a relaxation of the set
$\{Y \in \Sigma^{U \cup V}: \tr(Y_{UU}) = k, \; \tr(Y_{VV})=k\}$ and,
therefore, we are actually applying ADMM to solve a relaxation of \eqref{eq: WKDB relaxation}.
Here, the penalty parameter $\beta = \min \bbra{\max \bbra{ 5 n/k, 80},500}$ is used in the augmented Lagrangian.
As before, the subproblem to update $Y$ admits a closed-form solution using simplex projection,
and we update $X$ by applying the spectral projected gradient method to the dual subproblem
$$
	\min_{\lambda \ge 0, \phi \ge 0} \left\|\left[ 	\rbra{Y^{k+1} + \frac{U^k}{\beta} } - \Lambda \right]_+ \right\|_F^2  + \lambda^T \e + \phi^T \e - \frac{1}{2} \left\| Y^{k+1} + \frac{U^k}{\beta} \right\|^2_F,
$$
where
$$
	\Lambda := \mat{ {cc} \frac{1}{2}(\lambda \e^T + \e \lambda^T) & 0 \\ 0 & \frac{1}{2}(\phi \e^T + \e\phi^T) }.
$$	
Again, we stop the algorithm when both the relative duality gap and primal constraint violation are within
a desired error tolerance.

%-----------------------------------------------------------------------------------------------
% Description of the experiments.
%-----------------------------------------------------------------------------------------------
% Bernoulli clique trials.
\begin{figure}[t] 
	\caption{Number of recoveries for $N$-node graph with $k$ planted cliques of size at least $\hat r$ and $W$ generated
	according to the distributions $\Omega_1 = Bern(0.75),$ $\Omega_2 = Bern(p)$.
%	We plot the average number of recoveries of the planted cliques per set of $10$ trials for different minimum cluster sizes $\hat r$ and different
%	probabilities of adding noise edges.
	Brighter colours indicate a higher rate of recovery.}	
	\label{fig: KDC plots}
	\centering
	\subfloat[{$N=200$}]{\includegraphics[width=0.4\textwidth]{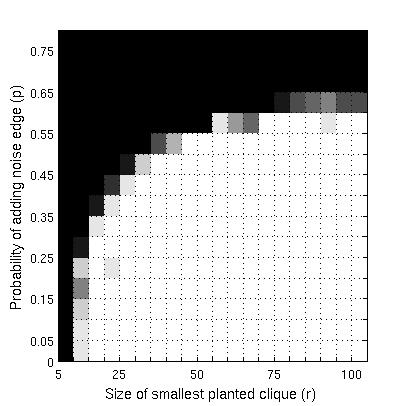} }  
	\subfloat[$N=500$]{\includegraphics[width=0.4\textwidth]{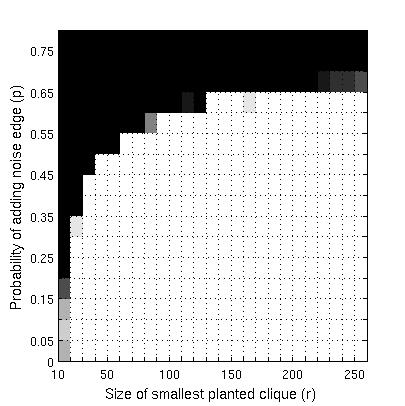}} 
\end{figure}

For $N=200$ and $N=500$ and a variety of choices of $\hat r \in  \{1, \dots, N\}$,  the following procedure  was repeated $10$ times.
We first partition the indices  $\{1, \dots, N\}$ into $k = \floor{N/\hat r}$  subsets $\{C_1, C_2, \dots, C_k\}$ of size at least $\hat r$.
We then  generate a random symmetric matrix $W \in \Sigma^n$ according to the planted cluster model with respect to  $\{C_1, C_2, \dots, C_k\}$ 
and one of two sets of probability distributions.
In the first, $W_{ij}$ is a Bernoulli random variable with probability of success $0.75$ if $i,j$ both belong to $C_\ell$ for some $\ell \in \{1,2,\dots, k\}$
and is a Bernoulli random variable with probability of success $p$ otherwise, for some fixed probability $p$.
In the second, each $W_{ij}$ is Gaussian with $\mu = \alpha$, $\sigma = 0.25$ for some  $\alpha \in [0.25, 1]$ if $i$ and $j$ belong to the same block, and $(\mu, \sigma^2) = (0.25, 0.25)$ otherwise.
For each choice of $p$ and $\alpha$, the ADMM procedure described above is called to approximately solve \eqref{E: w trace relax}.
In each experiment, the algorithm is terminated if the stopping criteria is achieved with error tolerance $\epsilon = 10^{-5}$ or after $500$ iterations, and
 the subproblem \eqref{eq: x subproblem} is solved to within  error tolerance  $\epsilon/10$ during each iteration.
% and regularization coefficient $\beta = 
 Let $X^*$ denote the optimal solution for \eqref{E: w trace relax} returned by the ADMM algorithm.
 We declare the block structure of $W$ to be successfully recovered if $\|X^* - X_0\|^2_F/\|X_0\|_F^2 < 10^{-3}$,
 where  $X_0$ is the proposed solution constructed according to \eqref{E: proposed sol}.
% $$	
% 	X_0 = \sum_{i=1}^k \frac{\x_i \x_i^T}{\|x_i\|^2}
%$$
%and $\x_i$ is the characteristic vector of $C_i$ for each $i=1,\dots, k$.

\begin{figure}[t] 
	\caption{Number of recoveries for $N$-node graph with $k$ planted cliques of size at least $\hat r$ and $W$ generated
	according to the distributions $\Omega_1 = N(\alpha, 0.25),$ $\Omega_2 = N(0.25, 0.25)$.}
	\label{fig: WKDC plots}
	\centering
	\subfloat[{$N=200$}]{\includegraphics[width=0.4\textwidth]{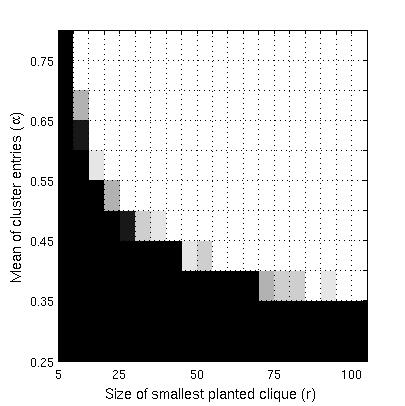} }  
	\subfloat[$N=500$]{\includegraphics[width=0.4\textwidth]{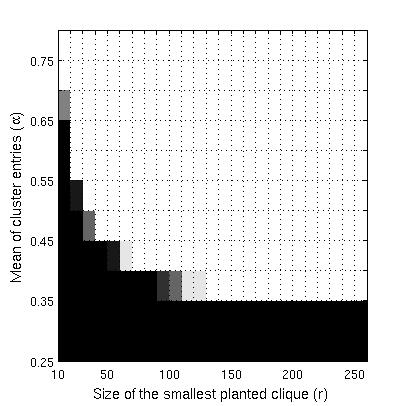}} 
\end{figure}

% Description of results.
Figures~\ref{fig: KDC plots}~and~\ref{fig: WKDC plots} display the average number of successes for each choice of $\hat r$ for $W$ sampled from the planted cluster model according
to the Bernoulli and Gaussian distributions, respectively. 
The empirical performance of our heuristic appears to match that predicted by Theorem~\ref{T: weighted kdc results}.
For each choice of $p$ or $\alpha$, there is a sharp phase transition between zero and perfect recovery as $\hat r$ increases past some threshhold.
It should be noted that the recovery guarantees given by Theorem~\ref{T: weighted kdc results} appear to be conservative compared to those observed empirically;
we have perfect recovery for values of $\hat r$ smaller than the left-hand side of \eqref{A: wkdc guarantee bound} for many trials.
Moreover, $W$ generated according to the Gaussian model do not necessarily satisfy the assumption $0 \le W_{ij} \le 1$ for $i,j \in \{1,\dots, N\}$.

We repeated the experiment for bipartite graphs drawn from the planted bicluster model.
For $(M, N) = (250, 200)$, and various minimum bicluster sizes $(\hat m, \hat n)$,
we randomly sample weight matrices $W$ from the planted bicluster model according to some
partition of $\{1,\dots, M\} \times \{1,\dots, N\}$ into $k = \floor{N/\hat n}$ bicliques with left and right vertex sets
of size at least $\hat m$ and $\hat n$ respectively.
For each $W$, we solve \eqref{eq: WKDB relaxation} using the ADMM algorithm described above and
declare the block structure to be recovered if the returned optimal solution $Z^*$ satisfies
$ \|Z^*-Z_0\|^2_F /\|Z_0\|_F^2 < 10^{-3}$, where
$Z_0$ is the proposed solution constructed according to \eqref{eq: WKDB sol}.
%$$
%	Z_0 = \sum_{i=1}^k \frac{\u_i\bv_i^T}{\|\u_i\| \|\bv_i\|}
%$$
%and $(\u_1, \bv_1), \dots, (\u_k, \bv_k)$ are the characteristic vectors of each block of $W$.
We plot the number of successful recoveries for each $(\hat m, \hat n)$ and generating distribution in Figure~\ref{fig: KDB plots}.
As before, the empirical behaviour of our heuristic reflects that predicted by Theorem~\ref{thm: biclustering recovery guarantee},
although there is some evidence that this theoretical recovery guarantee may be overly pessimistic.

\begin{figure}[t!] 
	\caption{Number of recoveries for $(250, 200)$-node bipartite graph with $k$ planted bicliques of size at least $(1.25 \hat n, \hat n)$
	and $W$ generated according to the distributions $\Omega_1$ and $\Omega_2$.}
	\label{fig: KDB plots}
	\centering
	\subfloat[{$\Omega_1 = Bern(0.75),$ $\Omega_2 = Bern(p)$}]{\includegraphics[width=0.4\textwidth]{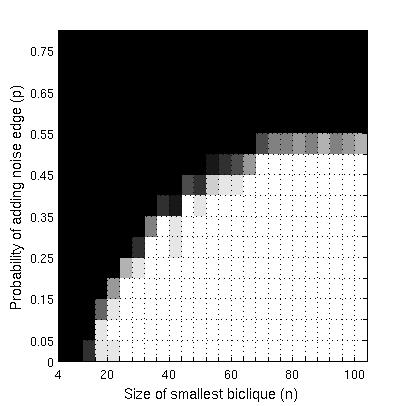} }  
	\subfloat[$\Omega_1 = N(\alpha, 0.25),$ $\Omega_2 = N(0.25, 0.25)$]{\includegraphics[width=0.4\textwidth]{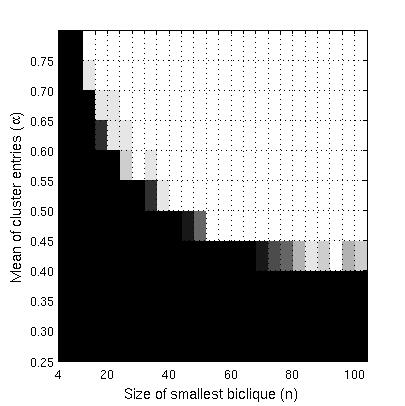}} 
\end{figure}

%-----------------------------------------------------------------------------------------------
% Summary of results.
%-----------------------------------------------------------------------------------------------

%%======================================================================

\section{Acknowledgements}
This research was supported in part by the Institute for Mathematics and its Applications with funds provided by the National Science Foundation,
and by a Postgraduate Scholarship from NSERC (Natural Science and Engineering Research Council of Canada).
I am grateful to Stephen Vavasis, Henry Wolkowicz, Levent Tun{\c{c}}el, Shai Ben-David, Inderjit Dhillon, Ben Recht, and
Ting Kei Pong for their helpful comments and suggestions. I am especially grateful to Ting Kei for his help implementing the ADMM algorithms used
to perform the numerical trials.
I would also like to thank Warren Schudy for suggesting some relevant references that were omitted
in an earlier version.
Finally, I thank the two anonymous reviewers, whose suggestions vastly improved the presentation and
organization of this paper.

%----------------------------------------------------------------------
% END MATERIAL
%----------------------------------------------------------------------

%******************************************************************************
% B I B L I O G R A P H Y
% -----------------------
\bibliographystyle{plain}

% This specifies the location of the file containing the bibliographic information.  
% It assumes you're using BibTeX (if not, why not?).
% The following statement causes the title "References" to be used for the biliography section:
%\renewcommand{\bibname}{References}
\bibliography{cliquebib,clusteringbib,background,nnmbib}   % name your BibTeX data base

% Tip 5: You can create multiple .bib files to organize your references. 
% Just list them all in the \bibliogaphy command, separated by commas (no spaces).
% Add the References to the Table of Contents

% The following statement causes the specified references to be added to the bibliography% even if they were not 
% cited in the text. The asterisk is a wildcard that causes all entries in the bibliographic database to be included (optional).
%\nocite{*}

%******************************************************************************
%%% Appendices.
\appendix
\cleardoublepage

%\chapter{Proofs of probabilistic bounds}

%\section{Proof of Theorem~\ref{thm: 2.4}}
%\label{App: 2.4}
%\input{A-thm-2-4-proof}

%\section{Proof of Theorem~\ref{T: lemma 4.1}}
%\label{App: Bern proof}
%\input{A-Lemma-4-1-proof}

%******************************************************************************
%\section{Appendix: Proof of Theorem~\ref{T: weighted kdc results}}
%\label{App: WKDC proof}

%\section{Appendix: Proof of Theorem~\ref{thm: Bernstein}}\label{sec: thm 2.4 proof}
%\input{Bernstein-proof.tex}

%******************************************************************************

%\section{Proof of Theorem~\ref{thm: Bernstein}}
%\label{App: Bernstein}
%\input{Source/Bernstein-proof}

\end{document}